\documentclass[reqno, 10pt]{amsart}

\usepackage[utf8]{inputenc}
\usepackage[english]{babel}

\usepackage{amsmath}
\usepackage{amsfonts}
\usepackage{amssymb}
\usepackage{dsfont}
\usepackage{amsthm}
\usepackage{graphicx}
\usepackage{bbding}

\usepackage[percent]{overpic}

\usepackage{hyperref}
\usepackage{cleveref}
\allowdisplaybreaks

\usepackage[numbers]{natbib}

\theoremstyle{remark}
\newtheorem{remark}{Remark}[section]

\theoremstyle{plain}
\newtheorem{theorem}{Theorem}[section]

\theoremstyle{remark}

\newcommand\numberthis{\addtocounter{equation}{1}\tag{\theequation}}
\numberwithin{equation}{section}

\title{The last zero-crossing of an Iterated Brownian motion with drift}
\author{F. Iafrate and E. Orsingher
}
\date{\today}

\begin{document}
	
	\nocite{*}
	
	\begin{abstract}
		In this paper we consider the iterated Brownian motion 
		$ ^{\mu_1}_{\mu_2}\!I(t) = B_1^{\mu_1} ( | B_{2}^{\mu_2} (t)|) $
		where $B_j^{\mu_j} , j=1,2$ are two independent Brownian motions with drift $\mu_j$.
		Here we study the last zero crossing before the maximum time span traveled by the inner process   of 	$ ^{\mu_1}_{\mu_2}\!I(t) $ and for this 
		purpose we derive the last zero-crossing distribution of the drifted Brownian motion. 
		We derive also the joint distribution of the last zero crossing before $ t $ and of the 
		first passage time through the zero level of a Brownian motion with drift $ \mu $ after 
		$ t $. All these results permit us to derive explicit formulas for 
		${^I _\mu T_0} = \sup \{ s < \max_{0\leq z\leq t} |B_2(z)| : B_1^\mu (s) = 0  \}$. 
		Also the iterated zero-crossing $  {^{\mu_1} T}_{0, {^{\mu_2} T}_{0,t}} $ is analyzed 
		and extended to the case where the level of nesting is arbitrary. 

	\end{abstract}

	\maketitle

\section{Introduction}


In this paper we consider the iterated Brownian motion 
\begin{equation}\label{eq:def-iter}
^{\mu_1}_{\mu_2}\!I(t) = B_1^{\mu_1} ( | B_{2}^{\mu_2} (t)|)
\end{equation}
where $B_j^{\mu_j} , j=1,2$ are two independent Brownian motions with drift $\mu_j$. 

There are other definitions of the iterated Brownian motion (see, for example, De Blassie (2004)) 
constructed by means of three independent Brownian motions as follows. Let
\begin{equation}
X_t = \left \{
\begin{aligned}
&B_1^{\mu_1} ( t)  \qquad & t \geq 0 \\
&B_2^{\mu_1} (- t)  \qquad & t < 0 \\
\end{aligned}
\right .
\end{equation}
be a two sided Brownian motion and then
\begin{equation}\label{eq:def-iter-2}
X_t( B_3^{\mu_2})
\end{equation}
gives the related iterated Brownian motion.  The process \eqref{eq:def-iter-2} can be 
interpreted as the diffusion of gas in a crack as described for example in \citet{deblassie2004}.

A similar composition given by  \citet{funaki1979} has the following form
\begin{equation}\label{eq:def-funaki}
Z_t = \left \{
\begin{aligned}
&B_2(B_1(t) ) \qquad & t \geq 0 \\
&i B_2 (- B_1(t) )  \qquad & t < 0 \\
\end{aligned}
\right . \enspace.
\end{equation}
The motions $B_1,\,B_2$ appearing in \eqref{eq:def-funaki} are independent, driftless 
and starting from zero. The law $u(x,t)$ of \eqref{eq:def-funaki} has been proved to satisfy the
fourth-order heat equation 
\[ 
\left\{ 
\begin{aligned} 
 &\frac{\partial u}{\partial t} = \frac{1}{2^3}\frac{ \partial ^4 u}{\partial x^4} & \qquad &
 	(x,t)  \in \mathbb R \times [0,\infty) \\
&u(x,0) = \delta (x)  & &x \in \mathbb R
\end{aligned}
\right. \enspace.
\]
An analogous composition of Brownian motions was studied in \citet{Hochberg1996}
where it is shown that its law satisfies the fourth order heat equation
\[ 
\left\{ 
\begin{aligned} 
&\frac{\partial u}{\partial t} = - \frac{1}{2^3}\frac{ \partial ^4 u}{\partial x^4} & \qquad &
(x,t)  \in \mathbb R \times [0,\infty) \\
&u(x,0) = \delta (x)  & &x \in \mathbb R
\end{aligned}
\right. \enspace
\]
the solution of which can be represented as 
\[ 
u_4(x,t) = \mathbb E \left( \frac{1}{\sqrt { 2 \pi |B(t) |  }}  \cos \left(  \frac{x^2}{2 B(t)} - \frac \pi 4 \right)   \right)
\enspace 
\]
%
%
%
%
%
(see \citet{Orsingher2011}).
The iterated Brownian motion has been examined from many points of view including its connection 
with fractional equations and some probabilistic properties as the iterated logarithm law (\citet{Khoshnevisan-lew1996}). 

We here consider the iterated Brownian motion \eqref{eq:def-iter} the distribution 
of which reads 
\begin{equation}\label{eq:iter-dist}
P\big( ^{\mu_1}_{\mu_2}\!I(t) \in \mathrm d x\big) / \mathrm d x =
\int_{0}^{\infty} \frac{
	e^{ - \frac{ (x-\mu_1s)^2}{2s} } 
	}{\sqrt{2\pi s}} 
	\bigg[ 
	\frac{
	e^{ - \frac{ (s-\mu_2t)^2}{2t} } 
	}{\sqrt{2\pi t}} + 
	\frac{
	e^{ - \frac{ (s+\mu_2t)^2}{2t} } 
	}{\sqrt{2\pi t}}
	 \bigg]
	 \mathrm d s \enspace .
\end{equation}
The Laplace-Fourier transform of \eqref{eq:iter-dist} has the form 
\begin{align*} \label{eq:it-fou-lap}
H(\gamma, \lambda) &= \int_{-\infty}^{+\infty} e^{i\gamma x} \int_{0}^{\infty} 
	e^{-\lambda t} P\big( ^{\mu_1}_{\mu_2}\!I(t) \in \mathrm d x\big) \, \mathrm d t  \numberthis \\
	&=
	\frac{1}{\sqrt{2\lambda + \mu_2^2}} 
	\bigg[ 
	\frac{
		1
		}{
		\frac{\gamma^2}{2} - i \gamma \mu_1 + \sqrt{2 \lambda  + \mu_2^2} - \mu_2
		}  \\
		& \qquad +
	\frac{  
		1
		}{
		\frac{\gamma^2}{2} - i \gamma \mu_1 + \sqrt{2 \lambda  + \mu_2^2} + \mu_2
		}
	\bigg] \enspace 
\end{align*}
as a direct check immediately shows. 
For $\mu_2 = 0$, the Fourier-Laplace transform \eqref{eq:it-fou-lap} reduces to 
\begin{equation}\label{eq:it-fou-lap-0}
H(\gamma, \lambda) = \frac{1}{\sqrt{\lambda}} 
\frac{1}{\sqrt{\lambda} + \frac{\gamma^2}{2^{3/2} } - i\gamma \frac{\mu_1}{2^{1/2} } } \enspace .
\end{equation}
Formula \eqref{eq:it-fou-lap-0} permits us to conclude that the distribution 
$p(x,t)$ of $^{\mu_1}_{0}\!I(t) = B_1^{\mu_1}( |B_2(t)| ) $ satisfies the fractional equation
\begin{equation}\label{eq:it-drift-eq}
\frac{\partial^{\frac 12}p } {\partial t^{\frac 12} } = \frac{1}{2^{\frac 32} } 
	\frac{\partial ^2 p}{\partial x^2} - \frac{\mu_1}{2^{\frac 12} } \frac{\partial p}{\partial x}
\end{equation}
with initial condition $p(x,0) = \delta(x)$. The time fractional derivative appearing in \eqref{eq:it-drift-eq} 
must be understood in the sense of Caputo.
%
%
%
%
%
The law $p$ of $^{\mu_1}_{0}\!I(t) $ satisfies also the fourth-order
non-homogeneous equation
\begin{equation}
\frac{\partial p}{\partial t} = \frac{1}{2^3} \frac{\partial^4 p}{\partial x^4} 
 - \frac{\mu_1}{2} \frac{\partial ^3 p}{\partial x^3}  + 
 \frac{\mu_1}{2} \frac{\partial^2 p}{\partial x^2} + 
 \frac{1}{2 \sqrt{2 \pi t}} \bigg( 
 \frac{\partial^2 \delta(x)}{\partial x^2} - \mu_2 \frac{\partial \delta(x)}{\partial x}
 \bigg) \enspace .
\end{equation}
We here study the distribution of the last visit to zero of the iterated Brownian motion \eqref{eq:def-iter}. We define the r.v.
\begin{equation}\label{eq:t0-iter-def}
_{\mu_2}^{\mu_1} \!T_{0,t} = \sup\Big\{ s < \max_{0\leq z\leq t } |B_2^{\mu_2}(z)| : B_1^{\mu_1}(s) = 0 \Big\}
\end{equation}
and represents the last zero-crossing of the outer Brownian motion $B_1^{\mu_1}$ with 
respect to a random time horizon which depends on a second independent Brownian motion $ B_2^{\mu_2} $. 

The presence of drift in the Brownian motions seriously complicates the derivation of the distribution. 
Therefore we have to study
\begin{equation}\label{eq:t0-drift-def}
^{\mu}T_{0,t} = \sup\{ s <t : B^{\mu}(s) = 0 \}
\end{equation}
and our result is that 
\begin{equation}\label{eq:t0-drift-df}
P\big(^{\mu}T_{0,t} <a \big) = 1 - \frac 2 \pi \int_{0}^{\sqrt \frac{t-a}{a}} 
\frac{  e^{ - \mu^2 \frac{a}{2} (1+y^2) }}{1 + y^2} \,\mathrm d y \enspace .
\end{equation}
The density of $ ^{\mu}T_{0,t} $ can be written as a weighted arcsine law and has the following structure
\begin{equation}
P\big(^{\mu}T_{0,t} \in\mathrm d a \big) / \mathrm d a = \frac{
	e^{ - \frac{\mu^2 t}{2} }
}{\pi \sqrt{ a (t-a)}} + 
\frac{\mu^2} {2\pi} \int_{a}^{t } 
\frac{ e^{- \frac{\mu^2 y}{2}} 
}{
	\sqrt{a(y-a)}
} \, \mathrm d y\enspace .
\end{equation}
Result \eqref{eq:t0-drift-df} includes the arcsine law of the last zero-crossing of a 
driftless Brownian motion that is
\begin{equation}
P\big(T_{0,t} <a \big) = \frac 2 \pi \arcsin \sqrt \frac at \qquad a<t \enspace .
\end{equation}
%
%
%
%
%
%
The form of the distribution of $ {^\mu}T_{0,t}  $ shows that the drift, independently
from its sign, makes the last zero-crossing occur earlier than in the driftless case. 
For the joint distribution of $({^\mu}T_{0,t}, {^\mu}T_{t,0})$ that is of the last
zero-crossing before $t$ and the first passage time through zero after $t$ we 
have that
\begin{equation}\label{eq:t0-t+-joint-intro}
P\big( {^\mu}T_{0,t} \in \mathrm d a, {^\mu }T_{t,0} \in \mathrm d b \big) = 
\frac{  e^{ - \frac{\mu^2b}{2} } }  {  2\pi \sqrt{ a (b-a)^3}  }\, \mathrm d a\,\mathrm d b
\qquad 0 < a<t<b<\infty \enspace .
\end{equation}
From \eqref{eq:t0-t+-joint-intro} we infer that the conditional distribution 
\begin{equation}\label{eq:t0-t+-cond-intro}
P(  {^\mu}T_{0,t} \in \mathrm d a | {^\mu}T_{t,0} = b)=
\frac 12 \frac{b}{\sqrt{at}}
\sqrt \frac{b-t}{(b-a)^3} \, \mathrm d a 
\end{equation}
which is not affected by the drift $ \mu $. We provide also the distribution 
\[ P( {^\mu}T_{t,0} \in \mathrm db  |  {^\mu}T_{0,t} = a) \] which converges
to the well-known result of \citet{ito1974diffusion} for $ \mu\to 0 $.

In Section 4 we give the distribution of \eqref{eq:t0-iter-def}, which has  the form
\begin{align}
&P( _{\mu_2} ^{\mu_1} T_{0,t} < a)  \\\notag 
&= 	P(\max_{0\leq z \leq t} | B^{\mu_2}(z)| < a) + 
 	\int_{a}^{\infty} P(^{\mu_1}\! T_{0,w} < a) 
	P(\max_{0\leq z \leq t} |B^{\mu_2}(z)| \in \mathrm d w )  
	 \,\, .
\end{align}
The joint distribution of $(B^\mu(t), \max_{0\leq s \leq t}  B^\mu(s), \min_{0\leq s \leq t}  B^\mu(s) )$
is given by 
\begin{align*}
&
P ( B^\mu(t) \in \mathrm d y, \alpha <  \min_{0\leq s \leq t}  B^\mu(s)  < \max_{0\leq s \leq t}  B^\mu(s) < \beta  ) 
\numberthis \label{eq:tri-dist-drift}\\
&=
\frac{\mathrm d y}{\sqrt{2\pi t}} \bigg[
\sum_{k=-\infty}^{\infty} \Big(
e^{ - \frac{ (y - 2k(\beta -\alpha) )^2}{2t}} - 
e^{ - \frac{ (y - 2k(\beta -\alpha) -2\alpha )^2}{2t}}
\Big) 
e^{ - \frac{\mu^2t}{2} + \mu ( y - 2k ( \beta - \alpha))}
\bigg] \\
&=
\frac{\mathrm d y}{\sqrt{2\pi t}} \bigg[
\sum_{k=-\infty}^{\infty} \Big(
e^{ - \frac{ (y - 2k(\beta -\alpha) - \mu t )^2}{2t}} - 
e^{2 \mu \alpha}
e^{ - \frac{ (y - 2k(\beta -\alpha) -2\alpha - \mu t)^2}{2t}}
\Big) 
\bigg] \enspace .
\end{align*}
%
%
%
%
%
%
%
The distribution of $\max_{0\leq z \leq t } | B^{\mu_2} (z)|$ is obtained from 
\eqref{eq:tri-dist-drift} by integrating w.r.t. $y$ in the interval $(-w,w)$ 
as will be shown in detail below. We note that the distribution of the maximum 
of the iterated Brownian motion can be given as
\begin{align*}
P\Big( \max_{0 \leq z \leq t}  {^{\mu_1}_{\mu_2}\!I(t)} > \beta \Big) &=
P\Big( \max_{0 \leq z \leq t}   B_1^{\mu_1} ( | B_{2}^{\mu_2} (t)|)  > \beta \Big) 
\numberthis \label{eq:max-iter-def}\\
&=
P\Big( \max_{0 \leq z \leq \max_{0 \leq s \leq t} |  B_2^{\mu_2} (s)| }   B_1^{\mu_1} ( z)  > \beta \Big)  \enspace . 
\end{align*} 
For $\mu_1 = \mu_2 = 0$ the explicit distribution of \eqref{eq:max-iter-def} 
is given in \citet{orsingher09}.
The last part of Section 4 is devoted to the analysis of the iterated zero-crossing
times $ {^{\mu_1} T}_{0, {^{\mu_2} T}_{0,t}} $ 
representing the last zero-crossing of $ B^{\mu_1} $ before the instant at which 
$ B^{\mu_2} $ visits zero at the last time before $ t $, with   $ B^{\mu_j} $ 
independent drifted Brownian motions. 

For the $ n- $th iterated zero-crossing time we obtain explicit distributions and study 
their mean-square convergence to zero.

	%
	\section{ Distribution of the last zero-crossing of a Brownian motion with drift}
%
%
%
%
%
%
%
The distribution of $^{\mu}T_{0,t} = \sup\{ s <t : B^{\mu}(s) = 0 \}$
is substantially different from the classical arcsine law of
$T_{0,t}$ of a driftless Brownian motion. 

The asymmetry of $B^\mu$ implies that also the distribution of
$^{\mu}T_{0,t}$ is asymmetric with respect to the instant $\frac{t}{2}$ and thus
\[ P( ^{\mu}T_{0,t} < \tfrac t2 ) > P( T_{0,t} < \tfrac t2 ) \]
and the difference between the left tail and the right tail 
of the distribution  of
$^{\mu}T_{0,t}$ increases for increasing values of the drift $\mu$.

The main result of this section is given in the following theorem.

\begin{theorem}\label{thm:t0-dist}
	For $ a < t<\infty $ we have that
	\begin{align}	 \label{eq:t0-drift-fr}
	P \{ ^{\mu}T_{0,t}  < a \} &=
	1 - \int_{0}^{t-a} 
	\frac{2}{2 \pi (a + s) } \sqrt{\frac as}
	e^{ - \frac{ \mu^2 }{2}(a+s) }  \, \mathrm ds 
	\\
	&=
	1 - \frac 2 \pi  
	\int_{0}^{\sqrt \frac{t-a}{a}} 
	\frac{ 
		e^{ - \frac{ \mu^2 a}{2} (1 + y^2) }
	}{1 + y^2} \, \mathrm d y   \notag  \\
	&=
	1 - \frac 2 \pi e^{ - \frac{ \mu^2 a}{2} } 
	\int_{0}^{\arccos \sqrt \frac{a}{t}} 
	e^{ - \frac{ \mu^2 a}{2} \tan^2 \theta }
	\, \mathrm d \theta \enspace .   \notag 
	\end{align}
	Furthermore
	\begin{align}\label{eq:t0-drift-dens}
	P\{ ^{\mu}T_{0,t} \in \mathrm d a \} / \mathrm d a &= 
	\frac{
		e^{ - \frac{\mu^2 t}{2} }
	}{\pi \sqrt{ a (t-a)}} + 
	\frac{\mu^2} \pi \int_{0}^{\sqrt \frac{t-a}{a} } e^{- \frac{\mu^2 a}{2}(1+y^2)} 
	\, \mathrm dy \\
	&=
	\frac{
		e^{ - \frac{\mu^2 t}{2} }
	}{\pi \sqrt{ a (t-a)}} + 
	\frac{\mu^2} {2\pi} \int_{a}^{t } 
	\frac{ e^{- \frac{\mu^2 y}{2}} 
	}{
	\sqrt{a(y-a)}
	}
\, \mathrm dy \notag \enspace .
\end{align}
\end{theorem}
%
%
%
%
%
%
%
\begin{proof}
	First we consider that
	\begin{align*}\label{eq:last-dist}
	& P\big( ^{\mu}T_{0,t} < a\big)  = \numberthis \\
	&= 
	\int_{0}^{\infty}
	\frac{
		e^{ - \frac { (y-\mu a)^2 }  { 2a } }
	}{
	\sqrt{ 2 \pi a }
}
P\{\min_{a \leq z \leq t} B^\mu (z) > 0 | B^\mu (a) = y\} \, \mathrm d y \quad  \\
& \qquad  +
\int_{-\infty }^{0}
\frac{
	e^{ - \frac { (y-\mu a)^2 }  { 2a } }
}{
\sqrt{ 2 \pi a }
}
P\{\max_{a \leq z \leq t} B^\mu (z) < 0 | B^\mu (a) = y\} \, \mathrm d y \enspace.
\end{align*}
Since, for $ y<0 $
\begin{align}\label{eq:max-needed}
& 
P\{ \max_{a \leq z \leq t} B^\mu (z) > 0 | B^\mu (a) = y  \} = \\
& = 
\int_{-\infty}^{\infty} 
P\{ \max_{0 \leq z \leq t-a } B^\mu (z) > 0, \, B^\mu(t-a) \in \mathrm d w | B^\mu (0) = y  \}
\nonumber
\end{align}
we need the joint distribution
\begin{align*}\label{eq:max-dist-init}
& 
P\{ \max_{0 \leq z \leq t} B^\mu (z) > \beta, \, B^\mu(t) \in \mathrm d w | B^\mu (0) = y  \} 
= \numberthis \\
& 
\left \{ 
\begin{aligned}
\frac{
	e^{  - \frac{(w-y)^2}{2t} } }{ \sqrt{2 \pi t} 
}
e^{ - \frac{ \mu^2 t}  {2}  + \mu (w-y)}  \qquad & w> \beta \\
\frac{
	e^{  - \frac{(2 \beta - w - y)^2}{2t} } }{ \sqrt{2 \pi t} 
}
e^{ - \frac{ \mu^2 t}  {2}  + \mu (w-y)} \qquad  & w< \beta 
\end{aligned}
\right.
\end{align*}
which can be obtained from that of the driftless case by applying the Girsanov 
theorem.

Thus, in view of \eqref{eq:max-dist-init} the distribution \eqref{eq:max-needed} becomes
%
%
%
%
%
%
%
%
\begin{align*}\label{eq:max-dist}
&
P\{ \max_{0 \leq z \leq t} B^\mu (z) > \beta | B^\mu (0) = y  \} \numberthis \\
& =
\int_{\beta}^{\infty}
\frac{
	e^{  - \frac{(w-y)^2}{2t} } }{ \sqrt{2 \pi t} 
}
e^{ - \frac{ \mu^2 t}  {2}  + \mu (w-y)} \, \mathrm d w + 
\int_{-\infty}^{\beta} 
\frac{
	e^{  - \frac{(2 \beta - w - y)^2}{2t} } }{ \sqrt{2 \pi t} 
}
e^{ - \frac{ \mu^2 t}  {2}  + \mu (w-y)} \, \mathrm d w \\ 
&=
\int_{\beta}^{\infty}
\frac{
	e^{  - \frac{(w-y)^2}{2t} } }{ \sqrt{2 \pi t} 
}
e^{ - \frac{ \mu^2 t}  {2}  - \mu y}
\big\{ e^{ \mu w }  + e^{ \mu (2\beta -w) } \big\}
\, \mathrm d w \\
& = 
P \{ T_\beta ^ \mu \leq t | B^\mu (0) = y\} =
\int_{0}^{t}  (\beta - y ) 
\frac{ e^{ - \frac{ (\beta - y - \mu s)^2} { 2s}  }}{\sqrt{ 2 \pi s^3} }
\, \mathrm d s \enspace .
\end{align*}
where $ T^\mu_\beta $ denotes the first passage time of a drifted Brownian motion through the level $ \beta $. The last step in \eqref{eq:max-dist} will be  proved below.

Analogously, for $ \beta < y $, we have that
\begin{align*}\label{eq:min-dist-init}
& 
P\{ \min_{0 \leq z \leq t} B^\mu (z) < \beta, \, B^\mu(t) \in \mathrm d w | B^\mu (0) = y  \} 
= \numberthis \\
& 
\left \{ 
\begin{aligned}
\frac{
	e^{  - \frac{(w-y)^2}{2t} } }{ \sqrt{2 \pi t} 
}
e^{ - \frac{ \mu^2 t}  {2}  + \mu (w-y)}  \qquad & w < \beta \\
\frac{
	e^{  - \frac{(2 \beta - w - y)^2}{2t} } }{ \sqrt{2 \pi t} 
}
e^{ - \frac{ \mu^2 t}  {2}  + \mu (w-y)} \qquad  & w > \beta \enspace .
\end{aligned}
\right. 
\end{align*}
Therefore the distribution of the minimum of the Brownian motion with drift becomes
\begin{align*}\label{eq:min-dist}
&
P\{ \min_{0 \leq z \leq t} B^\mu (z) < \beta | B^\mu (0) = y  \}  \numberthis\\
& =
\int_{-\infty}^{\beta}
\frac{
	e^{  - \frac{(w-y)^2}{2t} } }{ \sqrt{2 \pi t} 
}
e^{ - \frac{ \mu^2 t}  {2}  + \mu (w-y)} \, \mathrm d w + 
\int_{\beta}^{\infty} 
\frac{
	e^{  - \frac{(2 \beta - w - y)^2}{2t} } }{ \sqrt{2 \pi t} 
}
e^{ - \frac{ \mu^2 t}  {2}  + \mu (w-y)} \, \mathrm d w \\ 
&=
\int_{-\beta}^{\infty}
\frac{
	e^{  - \frac{(w+y)^2}{2t} } }{ \sqrt{2 \pi t} 
}
e^{ - \frac{ \mu^2 t}  {2}  - \mu( w + y) }
\, \mathrm d w 
  + 
\int_{-\beta}^{\infty}
\frac{
	e^{  - \frac{(w+y)^2}{2t} } }{ \sqrt{2 \pi t} 
}
e^{ - \frac{ \mu^2 t}  {2}  + \mu( 2 \beta +w  - y) }
\, \mathrm d w \enspace . 
\end{align*}
%
%
%
%
%
%
%
%
Now, by making the substitutions $\beta \mapsto - \beta, \,  y \mapsto -y,\, \mu \mapsto - \mu$ 
in the integrals \eqref{eq:max-dist}  we obtain the integrals  \eqref{eq:min-dist} 
and thus, from \eqref{eq:max-dist} we can write
\begin{align}\label{eq:min-dist-tb}
P\{ \min_{0 \leq z \leq t} B^\mu (z) < \beta | B^\mu (0) = y  \} &=
P \{ T_{-\beta} ^ {-\mu} \leq t | B^{-\mu} (0) = -y\} \\
& = 
\int_{0}^{t}  (- \beta + y ) 
\frac{ e^{  -\frac{ (- \beta + y + \mu s)^2} { 2s}  }}{\sqrt{ 2 \pi s^3} }
\, \mathrm d s \notag \quad .
\end{align}
In view of \eqref{eq:max-dist} and \eqref{eq:min-dist-tb} we obtain from
\eqref{eq:last-dist} that
\begin{align}\label{eq:last-dist-fun}
&
P \{ ^\mu T_{0,t}  < a \} = \\
&=
\int_{0}^{\infty} 
\frac{
	e^{ - \frac { (y-\mu a)^2 }  { 2a } }
}{
\sqrt{ 2 \pi a }
}
\bigg(
1 - 
\int_{0}^{t-a}  y  
\frac{ e^{ - \frac{ (  y + \mu s)^2} { 2s}  }}{\sqrt{ 2 \pi s^3} }
\, \mathrm d s
\bigg) \mathrm d y      \notag \\
& \qquad + 
\int_{-\infty}^{0} 
\frac{
	e^{ - \frac { (y-\mu a)^2 }  { 2a } }
}{
\sqrt{ 2 \pi a }
}
\bigg(
1 - 
\int_{0}^{t-a} (- y)  
\frac{ e^{ - \frac{ (   y + \mu s)^2} { 2s}  }}{\sqrt{ 2 \pi s^3} }
\, \mathrm d s
\bigg) \mathrm d y    \notag  \\
& =
1 - 
\int_{0}^{\infty} 
\frac{
	e^{ - \frac { (y-\mu a)^2 }  { 2a } }
}{
\sqrt{ 2 \pi a }
}
\int_{0}^{t-a}  y  
\frac{ e^{ - \frac{ (  y + \mu s)^2} { 2s}  }}{\sqrt{ 2 \pi s^3} }
\, \mathrm d s \,\,
\mathrm d y     \notag \\
& \qquad - 
\int_{-\infty}^{0} 
\frac{
	e^{ - \frac { (y-\mu a)^2 }  { 2a } }
}{
\sqrt{ 2 \pi a }
}
\int_{0}^{t-a} (- y)  
\frac{ e^{ - \frac{ (  y + \mu s)^2} { 2s}  }}{\sqrt{ 2 \pi s^3} }
\, \mathrm d s \,\,
\mathrm d y      \notag \\
&=
1 - 
\int_{0}^{\infty} 
\frac{
	e^{ - \frac { (y-\mu a)^2 }  { 2a } }
}{
\sqrt{ 2 \pi a }
}
\int_{0}^{t-a}  y  
\frac{ e^{ - \frac{ (  y + \mu s)^2} { 2s}  }}{\sqrt{ 2 \pi s^3} }
\, \mathrm d s \,\,
\mathrm d y     \notag \\
& \qquad - 
\int_{0}^{\infty} 
\frac{
	e^{ - \frac { (y+ \mu a)^2 }  { 2a } }
}{
\sqrt{ 2 \pi a }
}
\int_{0}^{t-a} y  
\frac{ e^{ - \frac{ (   y - \mu s)^2} { 2s}  }}{\sqrt{ 2 \pi s^3} }
\, \mathrm d s \,\,
\mathrm d y \notag \enspace . 
\end{align}
%
%
%
%
%
%
%
%
%
%
We need now to evaluate the following integral
\begin{align}
&\int_{0}^{\infty}
y
\frac{
	e^{ - \frac { (y - \mu a)^2 }  { 2a } }
}{
\sqrt{ 2 \pi a }
}
\frac{ e^{  -\frac{ (   y + \mu s)^2} { 2s}  }}{\sqrt{ 2 \pi s^3} }
\mathrm d y \\
& =
\frac{1}{2 \pi \sqrt{a s^3} }
\int_{0}^{\infty}
y
e^{ -\frac{ y^2} {2} \big( \frac{a+s }{as} \big)  }
e^{ - \frac{ \mu^2 }{2} (a+s) } \, \mathrm d y \notag  \\
&=
\frac{1}{2 \pi (a + s) } \sqrt{\frac as}
e^{ - \frac{ \mu^2 }{2}(a+s) }    \qquad s>0\,, \quad 0<a<t\notag \enspace .
\end{align}
The other integral in \eqref{eq:last-dist-fun} follows in the same way. 
In conclusion we have that
\begin{align}\label{eq:t0-df-2}
& 
P \{ ^\mu T_{0,t}  < a \}  = 1 - \int_{0}^{t-a} 
\frac{2}{2 \pi (a + s) } \sqrt{\frac as}
e^{ - \frac{ \mu^2 }{2}(a+s) }  \, \mathrm ds  \\
&=
1 - \frac 2 \pi 
\int_{0}^{\sqrt \frac{t-a}{a}} 
\frac{ 
	e^{ - \frac{ \mu^2 a}{2} (1 + y^2) }
}{1 + y^2} \, \mathrm d y    \notag \\
&=
1 - \frac 2 \pi e^{ - \frac{ \mu^2 a}{2} } 
\int_{0}^{\arctan \sqrt \frac{t-a}{a}} 
e^{ - \frac{ \mu^2 a}{2} \tan^2 \! \theta }
\, \mathrm d \theta    \notag \enspace . 
\end{align}
In order to prove that 
\begin{equation}\label{eq:min-dist-tb-2}
P\{ \min_{0 \leq z \leq t} B^\mu (z) < \beta | B^\mu (0) = y  \} =
P \{ T_{-\beta} ^ {-\mu} \leq t | B^{-\mu} (0) = -y\}
\end{equation}
we show that the $ \lambda$-Laplace transform of both members of \eqref{eq:min-dist-tb-2}
coincide for all values of $\lambda >0$.	
Thus for $ y > \beta $ we have
\begin{align*}
&
\int_{0}^{\infty} e^{-\lambda t} 
P\{ \min_{0\leq s \leq t } B^\mu (s) < \beta | B^\mu (0) = y \} \mathrm d t
\numberthis \label{eq:min-dist-tb-check}\\
&= 
\int_{0}^{\infty} e^{-\lambda t}
\int_{-\infty}^{\beta} \frac{ e^{ \frac{ (w-y)^2}{2t}}}{\sqrt{2 \pi t}} \,
e^{ - \frac{\mu^2t}{2} + \mu (w-y)} \, \mathrm d w \, \mathrm d t\\
& \qquad +
\int_{0}^{\infty} e^{-\lambda t}
\int_{\beta}^{\infty} \frac{ e^{ \frac{ (2\beta - w-y)^2}{2t}}}{\sqrt{2 \pi t}} \,
e^{ - \frac{\mu^2t}{2} + \mu (w-y)} \, \mathrm d w \, \mathrm d t\\
&=
\int_{-\infty}^{\beta} e^{\mu (w-y)}
\frac{ e^{  -|w-y| \sqrt{2 \lambda + \mu^2} } }{\sqrt{2 \lambda + \mu^2}} \,
\, \mathrm d w \\
& \qquad + 
\int_{\beta}^{\infty} e^{\mu (w-y)}
\frac{ e^{ - |2\beta - w-y| \sqrt{2 \lambda + \mu^2} } }{\sqrt{2 \lambda + \mu^2}} \,
\, \mathrm d w \\
&=
\frac{ 1 }{\sqrt{2 \lambda + \mu^2}}
\bigg[ 
\int_{-\infty}^{\beta} e^{\mu (w-y)
	- (y-w) \sqrt{2 \lambda + \mu^2} }  \,
\, \mathrm d w \\
& \qquad + 
\int_{\beta}^{\infty} e^{\mu (w-y)-
	(w + y - 2 \beta ) \sqrt{2 \lambda + \mu^2} } \,
\, \mathrm d w 
\bigg ]\\
&=
\frac{ e^{-y( \mu +  \sqrt{2 \lambda + \mu^2}) } }{\sqrt{2 \lambda + \mu^2}}
\bigg[ 
\frac{ e^{\beta ( \mu +  \sqrt{2 \lambda + \mu^2}) } }
{ \mu +  \sqrt{2 \lambda + \mu^2}} -
\frac{ e^{2\beta \sqrt{2 \lambda + \mu^2} }
	e^{\beta ( \mu -  \sqrt{2 \lambda + \mu^2}) } }
{ \mu - \sqrt{2 \lambda + \mu^2}} 	
\bigg] \\
&=
\frac{ e^{(\beta -y )( \mu +  \sqrt{2 \lambda + \mu^2}) } }{\lambda } \enspace . 
\end{align*}
In order to get rid of the absolute value in \eqref{eq:min-dist-tb-check}
we consider that $ \beta < y $ and, since in the first integral $ w<\beta $
 we have that $ |w-y | = y-w$. In the second integral 
$2 \beta - w- y = \beta - w + \beta -y$ and for $w > \beta$ we have that 
$ |2\beta -w-y| = w+y -2\beta $. 

In order to complete the proof of \eqref{eq:min-dist-tb-2}
we need the Laplace transform of the first passage time (formula \eqref{eq:min-dist-tb})
which becomes
\begin{align*}
&\int_{0}^{\infty} e^{- \lambda t} 
\int_{0}^{t} (-\beta + y ) 
\frac { e^{ - \frac{ (-\beta + y + \mu s)^2 }{2s}  } }{  \sqrt{ 2 \pi s^3}} 
\, \mathrm d s \, \mathrm d t \\
&=
\frac 1\lambda \int_{0}^{\infty} e^{-\lambda s} (-\beta + y)
\frac{
	e^{ - \frac {(-\beta + y)^2 } { 2s} }
}{\sqrt{2\pi s^3}} 
e^{ - \frac{\mu^2s}{2} - \mu (-\beta +y)}
\, \mathrm d s\ \\
&=
\frac{  e^{ (\beta - y) ( \mu + \sqrt{2\lambda + \mu^2})}}{\lambda } \enspace . 
\end{align*}
This shows that \eqref{eq:min-dist-tb-2} holds true. The same check can be
done for the distribution of the maximum of \eqref{eq:max-dist} exactly in 
the same way. 

In order to derive the density \eqref{eq:t0-drift-dens} it is convenient to use the second form of \eqref{eq:t0-drift-fr}
\end{proof}
\begin{remark}\label{rem:t0-drift-repr}
	The density \eqref{eq:t0-drift-dens} of $ ^\mu T_{0,t} $
	can be represented as an exponentially weighted arcsine law, 
	that is 
	\begin{equation}\label{eq:t0-dist-repr}
	P\big( ^\mu T_{0,t}  \in \mathrm d  a \big) / \mathrm d  a= 
	\mathbb E \bigg[ 
	\frac 1 \pi \frac{1}{\sqrt{
			a ( W-a)
		}}
		\cdot 	
		\mathds 1_{[W \geq a]}
		\bigg]
		\end{equation}
		where $W$ is a r.v. with absolutely continuous component
		in $(0,t)$ with density
		\[ f_W(w) = \frac{\mu^2}{2}  e^{ - \frac{\mu^2}{2}w}
		\qquad 0 <w<t
		\]
		and a discrete component at $w=t$ with  mass
		\[ P(W=t) = e^{ - \frac{\mu^2}{2}t} \enspace .  \]
	\end{remark}
	As a consequence of \eqref{eq:t0-dist-repr} the distribution function
	of $ ^\mu T_{0,t}  $ can be written as 
	\begin{align}\label{eq:t0-fr-repr}
	P(^\mu T_{0,t} < a) &= \int_{0}^{a} 
	\mathbb E \Big(  \frac{1}{\pi \sqrt{s(W-s)}}
	\mathds 1_{[W \geq s]} \Big) \, \mathrm d s \\
	&=
	\mathbb E \Big( \frac 2 \pi \arcsin \sqrt  \frac a W 
	\mathds 1_{[W \geq a]}  \Big)  + P(W<a) \notag  \enspace .
	\end{align}
	We give some details about \eqref{eq:t0-fr-repr}. 
	\begin{align*}
	&P(^\mu T_{0,t} < a) \\
	&=\int_{0}^{a} \mathrm d s \bigg[  
	\frac{
		e^{ - \frac{ \mu^2 t}{2} }
	}{\pi \sqrt{s (t-s)}} + \frac{\mu^2 }{2 \pi}
	\int_{s}^{t} 
	\frac{
		e^{ - \frac{ \mu^2 y}{2} }
	}{\pi \sqrt{s (y-s)}}\, 
	\mathrm d y
	\bigg] \numberthis \label{eq:t0-fdist-repr}\\
	&=
	e^{ - \frac{ \mu^2 t}{2} } 
	\frac 2 \pi \arcsin \sqrt  \frac st \bigg |_{s=0}^{s=a}
	\\
	&\qquad + \quad 
	\frac{\mu^2}{2} \bigg[
	\int_{0}^{a} \mathrm d s \int_{s}^{a} 
	\frac{
		e^{ - \frac{ \mu^2 y}{2} }
	}{\pi \sqrt{s (y-s)}}\, 
	\mathrm d y +
	\int_{0}^{a} \mathrm d s \int_{a}^{t} 
	\frac{
		e^{ - \frac{ \mu^2 y}{2} }
	}{\pi \sqrt{s (y-s)}}\, 
	\mathrm d y 	 
	\bigg] \\
	&=
	\frac 2 \pi \arcsin \sqrt \frac at e^{ - \frac{ \mu^2 t}{2} } 
	\\
	&\qquad + \quad 
	\frac{\mu^2}{2} \bigg[
	\int_{0}^{a} 
	e^{ - \frac{ \mu^2 y}{2} }
	\mathrm d y \int_{0}^{y} 
	\frac{	 \mathrm d s	 		
	}{\pi \sqrt{s (y-s)}} 
	+
	\int_{a}^{t} 
	e^{ - \frac{ \mu^2 y}{2} }
	\mathrm d y \int_{0}^{a} 
	\frac{	\mathrm d s 	 	 		
	}{\pi \sqrt{s (y-s)}}	 	 	 			 
	\bigg]	 \\
	&=
	\frac 2 \pi \arcsin \sqrt \frac at e^{ - \frac{ \mu^2 t}{2} } 
	\\ 	 	 &\qquad + \quad 
	\frac{\mu^2}{2} \bigg[
	\int_{0}^{a} 
	e^{ - \frac{ \mu^2 y}{2} }
	\mathrm d y  +
	\int_{a}^{t} 
	e^{ - \frac{ \mu^2 y}{2} }
	\frac  2 \pi \arcsin \sqrt \frac ay \mathrm d y 	 	 	 			 
	\bigg]	 \\
	&=
	\frac 2 \pi \arcsin \sqrt \frac at e^{ - \frac{ \mu^2 t}{2} } 
	+ 1 - e^{ - \frac{\mu^2a}{2}} 
	\\ 	 	 &\qquad + \quad 
	\Bigg[ - \frac  2 \pi e^{-\frac{\mu^2y}{2}} 
	 \arcsin \sqrt \frac ay 	 \Bigg]_{y=a}^{y=t}	 - 
	\frac 2 \pi \int_{a}^{t} 
	e^{- \frac{\mu^2y}{2} } \frac{ \sqrt a}{2 \sqrt {y^3}}
	\frac{1}{\sqrt{1 - \frac ay }}
	\mathrm d y \\
	&=
	1 - \frac 1 \pi \int_{a}^{t} 
	e^{- \frac{\mu^2y}{2} } \frac{ \sqrt a}{y}
	\frac{1}{\sqrt{y-a}}
	\mathrm d y \\
	&=
	1 - \frac 1 \pi \int_{0}^{t-a} 
	e^{- \frac{\mu^2(y+a)}{2} } \frac{ \sqrt a}{y+a}
	\frac{1}{\sqrt{y}}
	\mathrm d y  \enspace . 
	\end{align*}
	The fourth line of \eqref{eq:t0-fdist-repr} immediately shows that \eqref{eq:t0-fr-repr} holds true. 
	\begin{remark}
		For $t \to \infty $, from \eqref{eq:t0-drift-fr} we have that 
		\begin{equation}
		\lim_{t\to \infty} P(^\mu T_{0,t} < a) = 
		1 - \frac 2\pi \int_{0}^{\infty} 
		\frac{ e^{-\frac{\mu^2a}{2} (1+y^2) }}
		{1 + y^2} \mathrm d y \enspace .
		\end{equation}
		In \citet{gradtable}, page 338, we find formula 3.466.1
		\begin{equation}\label{eq:grad-ryzh}
		\int_{0}^{\infty} \frac{ e^{-\mu^2 x^2 }}
		{x^2 + \beta^2} \, \mathrm d x = 
		\frac{\pi e^{\beta^2 \mu^2}}{2\beta} 
		\bigg( 1 - \frac {2}{\sqrt \pi} \int_{0}^{|\beta \mu|}  
		e^{-w^2} \mathrm d w
		\bigg)
		\end{equation}
		valid for $ \beta>0\,,\,\,|\arg \mu| < \frac \pi 4 $.
		
		In our case we have $ \beta = 1 $ and $ \mu $ must be replaced by 
		$ \frac { \mu^2 a }{2} $ so that
		\begin{equation}\label{eq:t0-dist-infty}
		\lim_{t\to \infty} P(^\mu T_{0,t} < a) = 
		\frac{2}{\sqrt \pi} 
		\int_{0}^{|\mu|\frac{  \sqrt a} 2} e^{-w^2} \, \mathrm d w \qquad 
		0 < a < \infty \enspace .
		\end{equation}
		Result \eqref{eq:t0-dist-infty} shows that $ ^\mu T_{0,t} $, for $t\to \infty$ 
		is a random variable with  Gamma distribution with parameters $\nu = \frac{1}{2}$, 
		$ \lambda = \frac{\mu^2}{4}  $.
	\end{remark}
	\begin{remark}
		For small values of the drift $ \mu $ we have the following  approximation
		\begin{equation*}
		P(^\mu T_{0,t} \in \mathrm d  a) / \mathrm d a = 
		\frac{1}{\pi \sqrt{a (t-a)}} \Big( 1 + \frac{\mu^2}{2}(t-2a)\Big)
		\end{equation*}
		which shows that for $\mu^2 >0$ the distribution of $ ^\mu T_{0,t} $ is
		leftward skewed as the picture below shows. 
		\begin{figure}[h]
			\begin{overpic}[width=0.6\textwidth,tics=10]{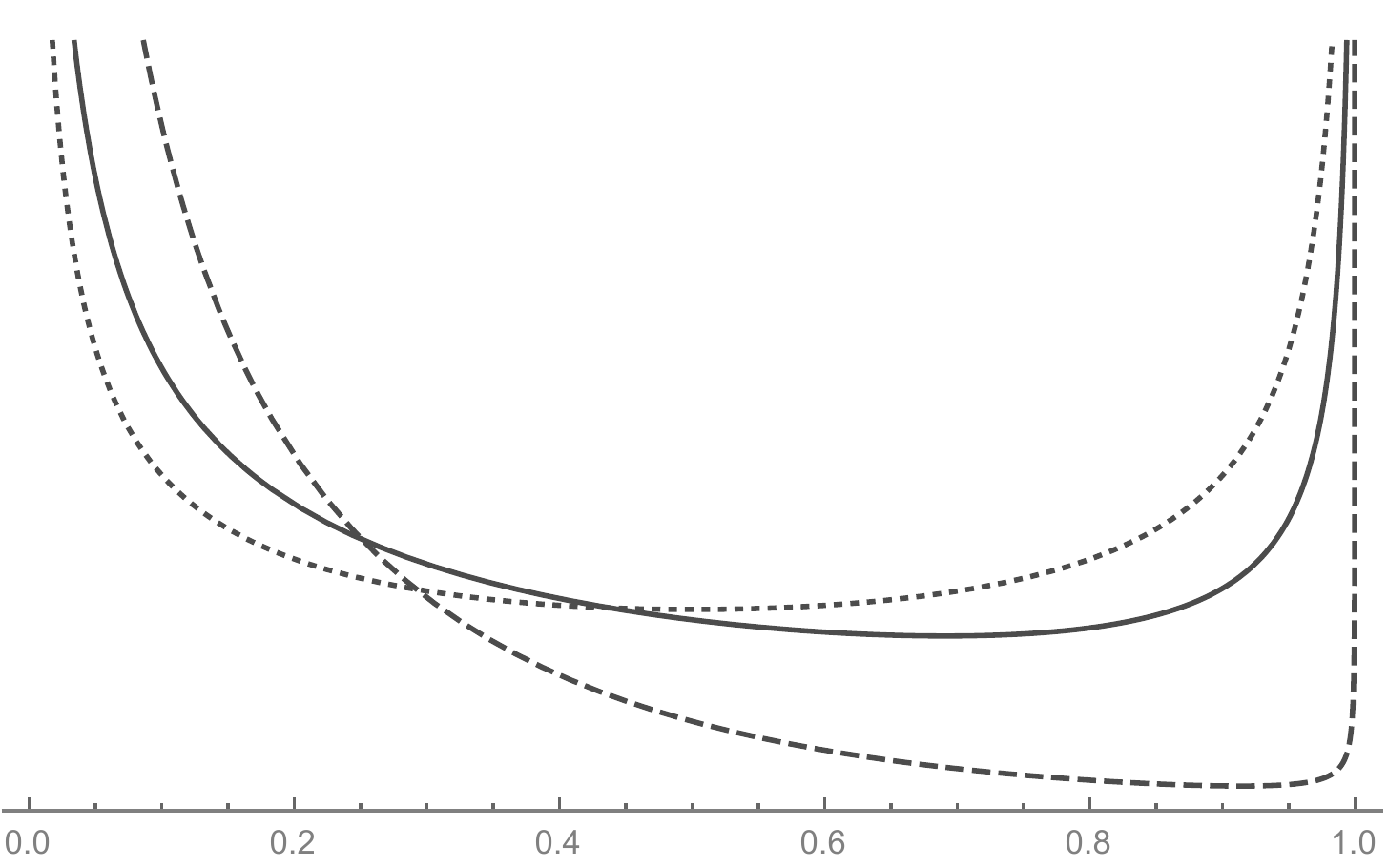}
				\put (100,-1) {$a$}
				\put (-12, 55)  {$f_{{^\mu\! T}}(a)$}
			\end{overpic}
			\centering
			\caption{Density of $^\mu T_{0,1} $ for $\mu= 0$ (dotted), for $\mu > 0$ (solid line),
				$\mu >\!>0$ (dashed). }
		\end{figure}

		For small values of $ \mu $ we have also the following approximation 
		of the distribution function of $ ^\mu T_{0,t}  $
		\begin{equation}
		P(^\mu T_{0,t} <  a)  \sim \frac 2 \pi \arcsin \sqrt \frac at 
		+ \frac { \mu^2 }{2} \sqrt {a (t-a)} \qquad 0 < a < t \enspace .
		\end{equation}
	\end{remark}
	In the next theorem we present the $ m $-th order moments of the
	r.v. $^\mu T_{0,t}   $ for all integer values of $ m $ and for finite $ t $.
	\begin{theorem}\label{thm:t0-em}
		Let $B^{\mu}(t), t \geq 0$ , $ \mu \in \mathds R $, be a Brownian motion with drift and 
		$  ^\mu T_{0,t} $ the random variable defined as in \eqref{eq:t0-drift-def}. 
		Then for any positive integer $m$ it holds that
		\begin{equation}\label{eq:t0-drift-em}
		\mathbb E [ ( ^\mu T_{0,t} )^m ] = \binom{2m}{m} \frac{m}{2^{2m}}
		\int_{0}^{t} a^{m-1} e^{- \frac{\mu^2 a}{2}} \mathrm d a \enspace .
		\end{equation}
	\end{theorem}
	\begin{proof}
		By applying the form \eqref{eq:t0-df-2} of the distribution of 
		$ ^\mu T_{0,t} $ we have that
		\begin{align*}
		\mathbb E [ ( ^\mu T_{0,t} )^m ]  &= 
		m \int_{0}^{t} a^{m-1} P(^\mu T_{0,t} > a) \, \mathrm d a \\
		&=
		\frac{m}{\pi} \int_{0}^{t} a^{m-1 + \frac 12} \, \mathrm d a 
		\int_{0}^{t-a} \frac{
			e^{- \frac { \mu^2}{2} (a+s)}
		}{\sqrt{ s} (a+s)} \, \mathrm d s\\
		&=
		\frac m \pi
		\int_{0}^{t} a^{m-1} e^{ - \frac{  \mu^2a}{2} } \, \mathrm d a 
		\int_{0}^{\frac ta - 1}
		\frac{ 
			e^{- \frac { \mu^2 a}{2} y}
		}{\sqrt{ y} (1 + y)} \, \mathrm d y \\
		&=
		\frac m \pi
		\int_{0}^{t} a^{m-1} e^{ - \frac{  \mu^2a}{2} } \, \mathrm d a 
		\int_{1}^{\frac ta}
		\frac{ 
			e^{- \frac { \mu^2 a}{2} (y-1)}
		}{\sqrt{ y-1} \,y} \, \mathrm d y \\
		&=
		\frac m \pi
		\int_{1}^{\infty} 
		\frac{ 
			\mathrm d y
		}{y \sqrt{ y-1} }
		\int_{0}^{\frac ty}
		a^{m-1} e^{- \frac { \mu^2 a}{2} y}
		\, \mathrm d a \\
		&=
		\frac m \pi
		\int_{1}^{\infty} 
		\frac{ 
			\mathrm d y
		}{y^{m+1} \sqrt{ y-1} }
		\int_{0}^{t}
		a^{m-1} e^{- \frac { \mu^2 a}{2} }
		\, \mathrm d a \\
		&=
		\frac m \pi
		\frac{\Gamma( m + \frac 12) \Gamma( \frac 12)}{\Gamma(m+1)}
		\int_{0}^{t}
		a^{m-1} e^{- \frac { \mu^2 a}{2} }
		\, \mathrm d a\\
		\end{align*}
		Since 
		\[  \frac{\Gamma( m + \frac 12) \Gamma( \frac 12)}{\pi \,\Gamma(m+1)} = \binom{2m}{m} \frac{1}{2^{2m}}\]
		result \eqref{eq:t0-drift-em} immediately follows.
	\end{proof}
	\begin{remark}
		The mean value 
		\[ \mathbb E [  ^\mu T_{0,t}  ]  = \frac{1}{\mu^2}\Big( 1 - e^{ - \frac{\mu^2 t}{2}} \Big) \]
		for $\mu \to 0$ becomes
		\[ 
		\lim_{\mu \to 0}\mathbb E [  ^\mu T_{0,t}  ]  = \frac t2 = \mathbb E T_{0,t} 
		\]
		i.e. it gives the mean value of the last zero-crossing time of a standard Brownian motion $T_{0,t}$.  
		Actually the $m-$th moment of $ ^\mu T_{0,t} $ admits the following representation
		\begin{equation}
		\mathbb E ( ^\mu T_{0,t} )^m = \mathbb E (T_{0,t})^m \cdot m \int_0^1 y^{m-1} e^{- \frac{\mu^2t y}{2}} \,
		\mathrm d y
		\end{equation}
		where we recall that the $m-$th moment of the last zero-crossing of a standard Brownian motion is
		given by
		\[
		\mathbb E (T_{0,t})^m = \binom{2m}{m} \frac{t^m}{2^{2m}} 
		\]
		as can be easily checked by direct calculation. This shows again that the distribution
		of $^\mu T_{0,t}$ is shrinked towards zero because
		\begin{align*}
			\mathbb E ( ^\mu T_{0,t} )^m &= \mathbb E (T_{0,t})^m \cdot m \int_0^1 y^{m-1} e^{- \frac{\mu^2ty}{2}}
			 \, \mathrm d y \\
			&\leq 
			\mathbb  E (T_{0,t})^m \int_0^1 m \, y^{m-1} \, \mathrm d y
			 =
			 \mathbb  E (T_{0,t})^m
		\end{align*}
		i.e. all the integer order moments of the zero-crossing time of a Brownian motion
		with drift are smaller than the corresponding moments for a driftless Brownian motion. 
	\end{remark}
	The moment generating function 
	\[
	M(\gamma) = \mathbb E e^{\gamma ^\mu T_{0,t}  }
	\]
	can be written down in terms of confluent hypergeometric 
	functions as follows. 
	
	\begin{theorem}\label{thm:t0-mgf}
		The moment generating function $ M(\gamma) $ of $ ^\mu T_{0,t} $
		is given by 
		\begin{align}\label{eq:t0-mgf}
		M(\gamma ) &= \mathbb E \, _1F_1(\gamma W; \tfrac 12 , 1)\\ 
		&= e^{-\frac{\mu^2 t}{2}} {_1 F}_1(\gamma t; \tfrac 12 , 1) +
		\frac {\mu^2} 2 \int_{0}^{t} e^{- \frac{\mu^2a}{2}}
		{_1}F_1(\gamma a; \tfrac 12 , 1) \, \mathrm d a \notag
		\end{align}
		where 
		\begin{subequations}
			\begin{align}\label{eq:1f1-def}
			_1F_1(x;a,b) &= 
			\sum_{m=0}^{\infty} \frac{a^{(m)}}{b^{(m)}m!} x^m \\
			&=
			\frac{\Gamma(b)}{\Gamma(a) \Gamma(b-a)}
			\int_{0}^{1} s^{a-1} (1-s)^{b-a-1} e^{xs} \, \mathrm d s \,\,,\quad 0<a<b
			\label{eq:1f1-integral}
			\end{align}
		\end{subequations}
		and $W$ is a r.v. with distribution coinciding with that of
		\Cref{rem:t0-drift-repr}. 
	\end{theorem}
	\begin{proof}
		From \Cref{thm:t0-em} we can derive the moment generating 
		function of $ ^\mu T_{0,t}  $ as follows. 
		\begin{align*}
		\mathbb E e^{\gamma \,^\mu T_{0,t}}  &= 
		1 + \sum\limits_{m=1}^{\infty} \frac{\gamma^m}{m!} \frac { m}{\sqrt \pi} 
		\frac{\Gamma(m + \frac 12)}{m!} 
		\int_{0}^{t} a^{m-1} e^{- \frac{\mu^2a}{2} } \, \mathrm d a \numberthis \\
		&=
		1 + \frac 1{\sqrt{\pi}} \int_{0}^{t} e^{- \frac{\mu^2a}{2} } \,
		\frac{\mathrm d }{\mathrm d a } \bigg\{ 
		\sum_{m=0}^{\infty} (a\gamma)^m \frac{\Gamma(m + \frac 12)}{m!^2}
		\bigg\} \, \mathrm d a \enspace .
		\end{align*}
		We recognize that 
		\begin{equation}\label{eq:1f1-reckon}
		\sum_{m=0}^{\infty}\frac{\Gamma(m + \frac 12)}{m!^2} 
		(a \gamma )^m = 
		{_1F_1}(\gamma a;\tfrac 12,1) \Gamma(\tfrac 12)
		\end{equation}
		because, in view of \eqref{eq:1f1-def} we have that
		\begin{align*}
		_1F_1(\gamma a;\tfrac 12,1) &= 
		\sum_{m=0}^{\infty}
		\frac{
			\frac 12 \Big( \frac 12 +1 \Big) \cdots  \Big( \frac 12 + m -1 \Big)
		}{1 \cdot 2 \cdots m \cdot m!} \, (\gamma a)^m\\
		&=
		\sum_{m=0}^{\infty}  \frac{\Gamma(m + \frac 12)}{m!^2}
		\frac{(a \gamma )^m }{\Gamma(\tfrac 12)} \enspace .
		\end{align*} 
		After an integration by parts, in view of \eqref{eq:1f1-reckon}, we have that
		\begin{align*}
		M(\gamma) &= 1 + \int_{0}^{t} e^{- \frac{\mu^2a}{2} } \,
		\frac{\mathrm d }{\mathrm d a } {_1F_1}(\gamma a;\tfrac 12,1) \, \mathrm d a\\
		&=
		e^{-\frac{\mu^2 t}{2}} {_1}F_1(\gamma t; \tfrac 12 , 1) +
		\frac {\mu^2} 2 \int_{0}^{t} e^{- \frac{\mu^2a}{2}}
		{_1F_1}(\gamma a; \tfrac 12 , 1) \, \mathrm d a \enspace .
		\end{align*}
		This concludes the proof of  \Cref{thm:t0-mgf}. 
	\end{proof}
	\begin{remark}
		
		By using the integral representation \eqref{eq:1f1-integral} of the $ {_1F_1}(x; a,b) $
		valid for $ a<b $ we have that:
		\begin{equation}\label{eq:1f1-repr-121}
		{_1F_1}(\gamma a; \tfrac 12 , 1)=
		\frac{1}{\pi }
		\int_{0}^{1} \frac{  e^{\gamma as} }   {    \sqrt{ s ( 1-s)}   }\, \mathrm d s \enspace .
		\end{equation}
		By plugging \eqref{eq:1f1-repr-121} into \eqref{eq:t0-mgf} we obtain
		\begin{align}\label{eq:mgf-alt}
		M(\gamma) &= 
		\frac{  e^{-\frac{\mu^2t}{2}}  }  {    \pi }  
		\int_{0}^{1} \frac{  e^{\gamma s t } }   {    \sqrt{ s ( 1-s)}   }\, \mathrm d s
		+ \frac{\mu^2}{2 \pi} \int_{0}^{t}  e^{-\frac{\mu^2a}{2}} 
		\, \mathrm d a 
		\int_{0}^{1} \frac{  e^{\gamma s a } }   {    \sqrt{ s ( 1-s)}   }\, \mathrm d s  \\
		&=
		\frac{  e^{-\frac{\mu^2t}{2}}  }  {    \pi }  
		\int_{0}^{t} \frac{  e^{\gamma s } }   {    \sqrt{ s ( t-s)}   }\, \mathrm d s
		+ \frac{\mu^2}{2} \int_{0}^{t}   e^{\gamma s } 
		\, \mathrm d s
		\int_{s}^{t} \frac{  e^{-\frac{\mu^2a}{2}} }   {  \pi  \sqrt{ s ( a-s)} }\, \mathrm da
		\notag    %
		\enspace .
		\end{align}
		The moment generating function \eqref{eq:mgf-alt} permits us to obtain
		the second form of the density of $ ^\mu T_{0,t} $ appearing in 
		\Cref{thm:t0-dist}. By simple calculation the two expressions of 
		\eqref{eq:t0-drift-dens} can be converted one into the other.

	\end{remark}

	%
	
\section{Joint Distributions}

Let $  {^\mu}T_{t,0} = \inf \{ s > t : B^{\mu} (s) = 0\} $ be the random variable describing
the first return to zero of a Brownian motion $ B^\mu $ after a fixed time instant $ t > 0$.
In this section we study the effect of  drift on the interplay between subsequent zero-crossing times. The main tool required to accomplish this task is the joint distribution of the random variables $ (  {^\mu}T_{0,t},  {^\mu}T_{t,0}) $ which we present in the following theorem. 
\begin{theorem}
	For $ a<t<b<\infty  $, the random vector $ (  {^\mu}T_{0,t},  {^\mu}T_{t,0}) $ has joint distribution
	given by
	\begin{subequations}
	\begin{align}\label{eq:t0-t+-joint-df}
	P( {^\mu}T_{0,t} <a \,,   {^\mu} T_{t,0} > b ) &= 
	1 - \frac 1\pi 
	\int_{0}^{b-a} \frac{\mathrm d s }{a+s}
		 \sqrt \frac as e^{- \frac{\mu^2}{2} (a+s)} \, \mathrm d s	\\
	\intertext{which admits a joint density function of the form} 
	P( {^\mu}T_{0,t} \in \mathrm d a \,,   {^\mu}T_{t,0} \in \mathrm d b ) &= 
	\frac{ e^{- \frac{\mu^2}{2}b} \mathrm d a \, \mathrm d b}{2\pi \sqrt{a (b-a)^3 }} 
	\label{eq:t0-t+-joint-dens} \enspace .
	\end{align}
	\end{subequations}
\end{theorem}
\begin{proof}
	The proof follows as in \Cref*{thm:t0-dist}. There the distribution of $ {^\mu}T_{0,t} $ 
	was derived by exploiting the equivalence between the following events
	\begin{equation*}
	\{   {^\mu}T_{0,t} < a \} = \{  B^\mu\text{ never crosses the zero level during } [a,t] \} 
	\end{equation*}
	Similarly, by replacing $ t $ with $ b $, the probability in \eqref{eq:t0-t+-joint-df} can 
	be restated as
	\begin{align*}
	P( {^\mu}T_{0,t} <a \,,   {^\mu} T_{t,0} > b ) &= 
	P(  B^\mu\text{ never crosses the zero level during } [a,b]) \\
	&\stackrel{\eqref{eq:t0-drift-dens}}{=}
	1 - \frac 1\pi 
	\int_{0}^{b-a} \frac{\mathrm d s }{a+s}
	\sqrt \frac as e^{- \frac{\mu^2}{2} (a+s)}  \enspace .
	\end{align*}
	Equation \eqref{eq:t0-t+-joint-dens} immediately follows.
\end{proof}
\begin{remark}
	We show that 
	\begin{equation}\label{eq:t+-inf}
	P ( {^\mu}T_{t,0} = \infty) = 2 \int_{0}^{|\mu|} \frac{ e^{- \frac{y^2}{2t}}}{\sqrt{2\pi t}} \, \mathrm d y \enspace .
	\end{equation}
	Since 
	\begin{equation*}
	P( {^\mu}T_{0,t} <a \,,   {^\mu} T_{t,0} < b ) = P( {^\mu}T_{0,t} <a )
	- P( {^\mu}T_{0,t} <a \,,   {^\mu} T_{t,0} > b )  
	\end{equation*}
	for $ a \to t $ and in view of \eqref{eq:t0-t+-joint-df} we have that
	\begin{align*}
	P ( {^\mu}T_{t,0} < \infty) &= \lim_{  \substack{a \to t \\ b \to \infty }}
			P( {^\mu}T_{0,t} <a \,,   {^\mu} T_{t,0} < b ) \\
	&=
	\frac 1\pi 
	\int_{0}^{\infty} \frac{\mathrm d s }{t+s}
	\sqrt \frac ts e^{- \frac{\mu^2}{2} (t+s)}  \\
	&=
	\frac 2\pi 
	\int_{0}^{\infty} \frac{\mathrm d y }{1 + y^2}
		e^{- \frac{\mu^2t}{2} (1+y^2)} \,  \\
	&=
	1 - 2 \int_{0}^{|\mu|} \frac{e^{- \frac{y^2}{2t}}}{\sqrt{2\pi t}}	\, \mathrm d y
	\end{align*}
	where in the last step we applied formula \eqref{eq:grad-ryzh}. Formula \eqref{eq:t+-inf} shows that
	the drifted Brownian motion $ B^\mu $ never crosses the zero level with positive probability in an 
	infinite time lapse. 
\end{remark}
\begin{remark}
	We now extract from \eqref{eq:t0-t+-joint-df} the marginal distributions of
	${^\mu}T_{t,0}$ and $ {^\mu}T_{0,t} $. This is important to prove that the results of
	this section are consistent with those of Section 2. 
	\begin{align*}\label{eq:t0-dist-check-joint}
	P( {^\mu}T_{0,t} < a ) &=
	P( {^\mu}T_{0,t} <a \,,   {^\mu} T_{t,0} > t ) \numberthis \\
	&=
	\int_{0}^{a} \mathrm d w 
	\int_{t}^{\infty} \frac{e^{- \frac{\mu^2 b}{2}}}{2 \pi \sqrt{w (b-w)^3}} \, \mathrm d b 
		 + P( {^\mu}T_{0,t} < a , {^\mu}T_{t,0} = \infty ) \\
	&=
	\int_{0}^{a} \bigg[
	\frac 1\pi \frac{e^{- \frac{\mu^2 t}{2}}}{\sqrt{w (t-w)}}  - 
		\frac{\mu^2}{2 \pi} \int_t^\infty  \frac{e^{- \frac{\mu^2 b}{2}}}{\sqrt{w (b-w)}} \, \mathrm d b
	\bigg] \, \mathrm d w\\ 
	& \qquad  + 
		1 - \frac 2\pi 
		\int_{0}^{\infty} \frac{\mathrm d y }{1 + y^2}
		e^{- \frac{\mu^2a}{2} (1+y^2)} 
	\end{align*}
	where $ P( {^\mu}T_{0,t} < a , {^\mu}T_{t,0} = \infty ) $ is obtained from \eqref{eq:t0-t+-joint-df}
	by letting $ b\to \infty $ and after having performed a suitable change of variable. 
	An integration by parts in \eqref{eq:t0-dist-check-joint} yields
	\begin{align*}
		P( {^\mu}T_{0,t} < a ) &= \frac 2\pi e^{ - \frac{\mu^2t}{2} } \arcsin \sqrt \frac at - 
			\frac{2 \mu^2}{2\pi} \int_{t}^{\infty} e^{- \frac{\mu^2 b }{2} } 
			\arcsin \sqrt \frac ab \, \mathrm d b \numberthis \\
			& \qquad + 1 - \frac 2\pi 
			\int_{0}^{\infty} \frac{\mathrm d y }{1 + y^2}
			e^{- \frac{\mu^2a}{2} (1+y^2)} \\
		&=	
		\int_{t}^{\infty} \frac{\sqrt a }{\pi b \sqrt{b-a}} e^{- \frac{\mu^2 b}{2}} \, \mathrm d b
		 + 1 - \frac 2\pi 
		 \int_{0}^{\infty} \frac{\mathrm d y }{1 + y^2}
		 e^{- \frac{\mu^2a}{2} (1+y^2)} \\
		&=
		1 - \frac 2\pi 
		\int_{0}^{\sqrt{\frac{t-a}{a}}} \frac{\mathrm d y }{1 + y^2}
		e^{- \frac{\mu^2a}{2} (1+y^2)} \qquad 0<a<t \enspace .
	\end{align*}
	The result obtained coincides with \eqref{eq:t0-drift-fr}.
\end{remark}
\begin{remark}
	We can easily obtain the marginal distribution of $  {^\mu}T_{t,0} $ in the following manner
	\begin{align*}
	P(  {^\mu}T_{t,0} \in \mathrm d b)/ \mathrm d b &= 
	\int_0^t \frac{e^{ - \frac{\mu^2 b}{2}}}{2 \pi \sqrt{a(b-a)^3}} \, \mathrm d a \numberthis \\
	&=
	\frac{e^{-\frac{\mu^2 b}{2}}}{\pi b} \int_{0}^{\arcsin \sqrt \frac tb } 
	\frac{\mathrm d \varphi}{\cos^{2} \varphi} \\
	&=
		\frac{e^{-\frac{\mu^2 b}{2}}
			}{\pi b } \sqrt \frac{t}{b-t}\,\,,\qquad b>t \enspace .
	\end{align*}
	Since 
	\begin{equation*}
	\int_{t}^{\infty} P(  {^\mu}T_{t,0} \in \mathrm d b ) = 
	\int_{t}^{\infty} \frac{e^{-\frac{\mu^2 b}{2}}
		}{\pi b } \sqrt \frac{t}{b-t} \,\mathrm d b =
		\int_0^\infty \frac{e^{- \frac{\mu^2 t}{2}  (1+y^2)}}{\pi (1+y^2)} \, \mathrm d y
	\end{equation*}
	by considering result \eqref{eq:t+-inf}
	we have that the distribution of $ {^\mu}T_{t,0} $ integrates
	to one. 
\end{remark}
\begin{remark}
	From the above results we have that 
	\begin{align}\label{eq:t0-cond-t+}
	P(  {^\mu}T_{0,t} \in \mathrm d a | {^\mu}T_{t,0} = b) &=
	\frac{
		\frac{
			e^{-\frac{\mu^2 b}{2} }
			}{2\pi \sqrt{a(b-a)^3}}
		}{
		 \frac{e^{-\frac{\mu^2 b}{2}}
		 }{\pi b } \sqrt \frac{t}{b-t} 
		} \, \mathrm d a 
		=
		\frac 12 \frac{b}{\sqrt{at}}
		\sqrt \frac{b-t}{(b-a)^3} \, \mathrm d a 
	\end{align}
	valid for $ a<t<b<\infty $.
	Note that the conditional distribution \eqref{eq:t0-cond-t+}
	coincides with the driftless case. 
	The density \eqref{eq:t0-cond-t+} tends to infinity as $a \to 0$  and converges to 
    $ \frac{b}{t(b-t)}$ as $a \to t$. Furthermore it displays a minimum at $a = \frac{b}{4}$
    which lies inside the support only for $b \leq 4t$. For $b\to \infty$ 
    \[
    P(  {^\mu}T_{0,t} \in \mathrm d a | {^\mu}T_{t,0} = b) \rightarrow \frac{\mathrm da}{2 \sqrt{at}} 
    \qquad 0 < a < t \enspace .
    \]

	The other conditional distribution instead depends on the 
	drift $ \mu $ since, in view of \eqref{eq:t0-drift-dens} 
	we have that 
	\begin{align}\label{eq:t+-cond-t0}
	P(  {^\mu}T_{t,0} \in \mathrm d b | {^\mu}T_{0,t} = a) &=
	\frac{
		\frac{
			e^{-\frac{\mu^2 b}{2} }
		}{2\pi \sqrt{a(b-a)^3}}
	}{
	\frac{
		e^{ - \frac{\mu^2 t}{2} }
	}{\pi \sqrt{ a (t-a)}} + 
	\frac{\mu^2} {2\pi} \int_{a}^{t } 
	\frac{ e^{- \frac{\mu^2 y}{2}} 
	}{
	\sqrt{a(y-a)} } \, \mathrm d y
	} \, \mathrm d b\\
	&=
	\frac{
		\frac 12 \sqrt \frac{t-a}{(b-a)^3} 
		e^{- \frac{\mu^2}{2} (b-t)}
		}{
		1 + \mu^2 \sqrt{a(t-a)} e^{\frac{\mu^2}{2}(t-a)}
		\int_{0}^{\sqrt \frac{t-a}{a}}
		e^{ - \frac{\mu^2 y^2 a}{2}}\, \mathrm d y
		} \, \mathrm d b\notag
\end{align}
For $ \mu=0 $ we have from \eqref{eq:t+-cond-t0} the well-known 
result 
\begin{equation*}
	P( T_{t,0} \in \mathrm d b | T_{0,t} = a) = 
	\frac 12 \sqrt \frac{t-a}{(b-a)^3} \, \mathrm d b \qquad 
	b>t>a \enspace .
\end{equation*}
We observe that
\begin{align*}\label{eq:t0-dens-t+-inf}
&
P( {^\mu}T_{0,t} \in \mathrm d a , {^\mu}T_{t,0} = \infty ) \numberthis \\
&=
P( {^\mu}T_{0,t} \in \mathrm d a )
- P( {^\mu}T_{0,t} \in \mathrm d a \,,  
 {^\mu} T_{t,0} < \infty )  \\
 &=
 \Bigg [
 \frac{
  e^{ - \frac{\mu^2 t}{2} }
	}{\pi \sqrt{ a (t-a)}} + 
	\frac{\mu^2} {2\pi} \int_{a}^{t } 
	\frac{ e^{- \frac{\mu^2 y}{2}} 
	}{
	\sqrt{a(y-a)} } \, \mathrm d y\Bigg ] \, \mathrm d a 
	-
	\int_{t}^{\infty} \frac{e^{- \frac{\mu^2 b}{2}}}{2 \pi \sqrt{a (b-a)^3}} \, \mathrm d b \, \mathrm d a\\
	&=
	\frac{\mu^2}{\pi} \int_{0}^{\infty} 
	e^{ - \frac{\mu^2 a}{2}(1+y^2)}\, \mathrm d y \, \mathrm d a
	=
	\frac{|\mu |}{\sqrt{2 \pi a}} e^{-\frac{\mu^2 a }{2}} 
	\, \mathrm d a \enspace 
\end{align*}
(see \citet{ito1974diffusion}, page 28, problem 3).

From \eqref{eq:t0-dens-t+-inf} we infer that
\begin{equation}
P( {^\mu}T_{t,0} = \infty | {^\mu}T_{0,t} = a) =
\frac{
	\sqrt \frac \pi 2 |\mu| \sqrt{t-a} 
	e^{ \frac{\mu^2}{2} (t-a) }
	}{
	1 + \mu^2 \sqrt{a(t-a)} e^{\frac{\mu^2}{2}(t-a)}
	\int_{0}^{\sqrt \frac{t-a}{a}}
	e^{ - \frac{\mu^2 y^2 a}{2}}\, \mathrm d y
	} \enspace .
\end{equation}
\end{remark}
It can be directly checked that 
\[ 
\int_t^{\infty} P (  {^\mu}T_{t,0} \in \mathrm d b | {^\mu}T_{0,t} = a ) 
+
P (  {^\mu}T_{t,0} = \infty | {^\mu}T_{0,t} = a ) = 1 \enspace .
\] 
Finally we give the distribution of the length of the interval straddling the point $t$ without zero-crossings is
\[
P(  {^\mu}T_{t,0} -  {^\mu}T_{0,t} \in \mathrm d w ) = 
\frac{e^{ - \frac{\mu^2 w}{2  } }}{ 2 \pi \sqrt{w^3 }} 
\int_{\max ( 0, t-w) }^{t} 
\frac{e^{ - \frac{\mu^2 a}{2  } }}{ \sqrt{a }} \,\mathrm d a \,\,\mathrm d w
\qquad 0<w<\infty \enspace .
\]

	%
	
\section{Zero-crossing of the iterated Brownian motion}

Our aim here is to study the distribution of the random time 
\begin{equation}\label{eq:t0-iter-def-2}
_{\mu_2}^{\mu_1} \!T_{0,t} = \sup\Big\{ s < \max_{0\leq z\leq t } |B_2^{\mu_2}(z)| : B_1^{\mu_1}(s) = 0 \Big\} \enspace .
\end{equation}
The random variable defined in \eqref{eq:t0-iter-def-2} represents the last zero-crossing of $ B_1^{\mu_1} $
before the maximal time span attained by the inner process $B_2^{\mu_2}$.

In this case the time horizon of $B_1^{\mu_1}$ is given by the random variable 
$ \max_{0\leq z\leq t } |B_2^{\mu_2}(z) |$ that is we must  consider the last crossing 
of the zero level performed by $ B_1 $ for the maximal value of the process $ B_2^{\mu_2} $ representing
the ``time'' in the iterated Brownian motion. 

In order to write down the distribution of \eqref{eq:t0-iter-def-2} we need to start from that of the triple r.v. 
$\big(B^\mu(t), \max_{0\leq s \leq t}  B^\mu(s), \min_{0\leq s \leq t}  B^\mu(s) \big) $
(on this point consult the book by \citet{shorack2009empirical}) which is given explicitly by 
\begin{align*}\label{eq:tri-dist-drift-2}
		&
		P ( B^\mu(t) \in \mathrm d y, \alpha <  \min_{0\leq s \leq t}  B^\mu(s)  < \max_{0\leq s \leq t}  B^\mu(s) < \beta  ) 
		\numberthis \\
		&=
		\frac{\mathrm d y}{\sqrt{2\pi t}} \bigg[
		\sum_{k=-\infty}^{\infty} \Big(
		e^{ - \frac{ (y - 2k(\beta -\alpha) )^2}{2t}} - 
		e^{ - \frac{ (y - 2k(\beta -\alpha) -2\alpha )^2}{2t}}
		\Big) 
		e^{ - \frac{\mu^2t}{2} + \mu ( y - 2k ( \beta - \alpha))}
		\bigg] \\
		&=
		\frac{\mathrm d y}{\sqrt{2\pi t}} \bigg[
		\sum_{k=-\infty}^{\infty} \Big(
		e^{ - \frac{ (y - 2k(\beta -\alpha) - \mu t )^2}{2t}} - 
		e^{2 \mu \alpha}
		e^{ - \frac{ (y - 2k(\beta -\alpha) -2\alpha - \mu t)^2}{2t}}
		\Big) 
		\bigg] 
\end{align*}
for $ \alpha < y < \beta $ and starting point equal to zero.

Each term in \eqref{eq:tri-dist-drift-2} can be viewed as
the distribution of a Brownian motion with drift on the
half line $ (\alpha, \infty) $ starting from the initial point
$ x = 2k(\beta -\alpha) $. The first version of \eqref{eq:tri-dist-drift-2} shows that the distribution of the 
triple 
$\big(B^\mu(t), \max_s  B^\mu(s), \min_s  B^\mu(s) \big) $
can be derived from that of the non-drifted Brownian motion by
a Girsanov argument. 
From \eqref{eq:tri-dist-drift-2} we extract the distribution 
\begin{equation}\label{eq:max-drift-dist}
P\Big( \max_{0\leq s \leq t}  |B^\mu(s)| <\beta \Big ) =
\int_{-\beta}^{\beta} \frac{\mathrm d y}{\sqrt{2\pi t}} 
\sum_{k=-\infty}^{\infty} f(r) 
e^{- \frac{(y-2\beta r - \mu t)^2}{2t}} 
\end{equation}
where 
\begin{equation*}
	f(r) = \begin{cases}
		(-1)^r e^{-2 \mu \beta} & r \text{ odd} \\
		(-1)^r  & r \text{ even} \enspace .
	\end{cases}
\end{equation*}
As a byproduct of the above formulas we have also that
\begin{equation}\label{eq:max-bridge-drift-dist}
P\Big( \max_{0\leq s \leq t}  |B^\mu(s)| <\beta 
\Big| B^{\mu}(t)=0 \Big)  =
\sum_{r} f(r) e^{-\frac 2t (\beta^2 r^2 + r \beta \mu t)} \enspace .
\end{equation}
For $ \mu=0 $ \eqref{eq:max-bridge-drift-dist} coincides with 
the well-known Kolmogorov-Smirnov distribution for the Brownian bridge.

For the non-drifted Brownian motion $ B $ the distribution of 
$  \max_{0\leq s \leq t}  |B(s)|  $ reduces to 
\begin{equation}
P\Big( \max_{0\leq s \leq t}  |B(s)| <\beta \Big ) =
\int_{-\beta}^{\beta} \frac{\mathrm d y}{\sqrt{2\pi t}} 
\sum_{k=-\infty}^{\infty} (-1)^r 
e^{- \frac{(y-2\beta r )^2}{2t}} \enspace .
\end{equation}
We pass now to the derivation of the last zero-crossing of the iterated Brownian motion. 
\begin{theorem}
	For the last zero-crossing of the iterated Brownian motion
	\[
	{^{\mu}I}(t) = B_1^{\mu} ( | B_{2} (t)|)
	\]
	that is for 
	\[
	{^I _\mu T_0} = \sup \Big\{ s < \max_{0\leq z\leq t} |B_2(z)| : B_1^\mu (s) = 0 \Big \}
	\]
	we have the following result 
	\begin{align*}\label{eq:t0-iter-dist}
			P ( { ^{I}_{\mu} T}_0  < a) &=
			\int_{0}^{\infty} 
			P( \sup\{ s < w : B_1^\mu (s) = 0 \} <a ) \, 
			P\Big(   \max_{0\leq z\leq t} |B_2(z)| \in \mathrm d w \Big ) \numberthis \\
			&=
			\frac{2}{\sqrt{2\pi}} \int_{0}^{|\mu| \sqrt{2a}}  e^{-\frac{w^2}{2}}
			\, \mathrm d w  \quad +  \\
			& \qquad +
			\frac{2 \sqrt a }{\pi}
			\int_{a}^{\infty} 
			\frac{
				e^{- \frac{\mu^2w}{2} }
			}{w \sqrt{w-a}} \, \mathrm d w
			\int_{-w}^{w} \frac{\mathrm d z}{\sqrt{2 \pi t}}
			\sum_{r=-\infty}^{\infty} (-1)^r e^{ - \frac{(z-2wr)^2}{2t} }
	\end{align*}
	for $ 0< a <\infty $.
\end{theorem}
\begin{proof}
	We start by splitting the integral in \eqref{eq:t0-iter-dist} into two subintervals $ (0,a) $
	and $ (a, \infty) $. 
	
	In the first case for $ w<a $, $ ( ^\mu T_{0,w}  <w) \subset ( ^\mu T_{0,w}  < a)  $ and $ ( ^\mu T_{0,w}  <w)  $
	occurs with probability one. Hence
	\begin{align*}
		&P( {^I _\mu T_0} <a ) \\ &=  
		\int_{0}^{a} P\Big( \max_{0 \leq z \leq t } | B (z)|  \in \mathrm d w \Big)   \\ 
		& \qquad + \int_a^{\infty} P( \sup\{s<w : B^{\mu}(s) = 0\} \, < \, a ) P\Big( \max_{0\leq z\leq t} | B(z) | 
		\in \mathrm d w\Big) \\
		&=
		\int_{0}^{a} P\Big( \max_{0 \leq z \leq t } | B (z)|  \in \mathrm d w \Big)    \\ 
		& \qquad + \int_{a}^{\infty} \mathrm d w 
		\bigg[1 - \frac 2\pi \int_{0}^{\sqrt \frac{w-a}{a}} 
		\frac{e^{ - \frac{\mu^2 a }{2}(1+y^2)  }}{1+y^2} \, \mathrm d y 
		\bigg] \, \, \frac{\mathrm d }{\mathrm d w} \int_{-w}^{w} 	\frac{\mathrm d z}{\sqrt{ 2\pi t} }
		\sum_{r=-\infty}^{\infty} (-1)^r e^{ - \frac{ (z-2wr)^2}{2t} } \\
		&=
		P\Big(  \max_{0\leq z \leq t}|B(z)| <a  \Big)   \\
		& \qquad + 
		\Big(1 - \frac 2\pi \int_{0}^{\sqrt \frac{w-a}{a}} 
		\frac{e^{ - \frac{\mu^2 a }{2}(1+y^2  )}}{1+y^2} \, \mathrm d y 
		\Big) \, \, 
		\int_{-w}^{w} 	\frac{\mathrm d z}{\sqrt{ 2\pi t} }
		\sum_{r=-\infty}^{\infty} (-1)^r e^{ - \frac{ (z-2wr)^2}{2t} }\Big|_{w=a}^{w=\infty} \\
		& \qquad +
		\int_{a}^{\infty} \mathrm d w \Big ( 
		\int_{-w}^{w} 	\frac{\mathrm d z}{\sqrt{ 2\pi t} }
		\sum_{r=-\infty}^{\infty} (-1)^r e^{ - \frac{ (z-2wr)^2}{2t} } \Big )
		\frac 2{\pi w} \sqrt \frac{a}{w-a} e^{ - \frac{\mu^2 w}{2}}\\
		&=		
		P\Big(  \max_{0\leq z \leq t}|B(z)| <a  \Big)  +  1 - 
		\frac 2\pi \int_{0}^{\infty} \frac{e^{ - \frac{\mu^2 a }{2}(1+y^2)  }}{1+y^2} \, \mathrm d y  - 
		P\Big(  \max_{0\leq z \leq t}|B(z)| <a  \Big)\\
		&\qquad + 
		\frac{2 \sqrt a }{\pi} \int_{a}^{\infty} \frac{ e^{  - \frac{\mu^2w}{2} }}{w \sqrt{w-a}} \, \mathrm d w
		\int_{-w}^{w} 	\frac{\mathrm d z}{\sqrt{ 2\pi t} }
		\sum_{r=-\infty}^{\infty} (-1)^r e^{ - \frac{ (z-2wr)^2}{2t} } \\ 
		&=		 
		\frac{2}{\sqrt{\pi}} \int_{0}^{|\mu| \sqrt{a}}  e^{-w^2}
		\, \mathrm d w  \quad   \\
		& \qquad +
		\frac{2 \sqrt a }{\pi}
		\int_{a}^{\infty} 
		\frac{
			e^{- \frac{\mu^2w}{2} }
		}{w \sqrt{w-a}} \, \mathrm d w
		\int_{-w}^{w} \frac{\mathrm d z}{\sqrt{2 \pi t}}
		\sum_{r=-\infty}^{\infty} (-1)^r e^{ - \frac{(z-2wr)^2}{2t} }
		\enspace .
	\end{align*}
	This completes the proof of \eqref{eq:t0-iter-dist}. 
	\begin{remark}
		First we note that 
		\begin{equation}\label{eq:t0-iter-dist-0}
		\lim_{a \to 0} P({^I _\mu T_0} <a ) = 0 \enspace .
		\end{equation}
		In order to prove \eqref{eq:t0-iter-dist-0} we need to prove that the following integral is finite
		\[ 
		\int_{0}^{\infty} \frac{e^{ - \frac{\mu^2w}{2} } }   {w\sqrt w}  
		P\Big(  \max_{0\leq z \leq t}|B(z)| <w  \Big) \, \mathrm d w<\infty  \enspace .
		\]
		We observe that
		\[ 
		P\Big(  \max_{0\leq z \leq t}|B(z)| <w  \Big) \leq 1- 2P(B(t) > w) 
		\]
		and thus 
		\begin{align*}
			&\lim_{w \to 0^+ } \frac{
				\frac{e^{ - \frac{\mu^2w}{2} } }   {w\sqrt w} 
				P\Big(  \max_{0\leq z \leq t}|B(z)| <w  \Big)
			}{
				\frac 1w
			} \\				
			&\leq 
			\lim_{w \to 0^+ } 
			\frac{e^{ - \frac{\mu^2w}{2} } }   {\sqrt w} 
			(1 - 2 P(B(t) > w) ) \\
			&=
			\lim_{w \to 0^+ } 
			\frac{
				1- 2 \int_{\frac{w}{\sqrt t}}^{\infty} \frac{e^{- \frac{y^2}{2}}}{\sqrt{2\pi}}
			}{\sqrt w}
			= 
			\lim_{w \to 0^+ }  2^2	\frac{
				\frac{e^{- \frac{w^2}{2t}} }{\sqrt{2\pi t}}
			}{w^{-\frac 12}}  = 0 \enspace .
		\end{align*}
		This permits us  to conclude that  \eqref{eq:t0-iter-dist-0} 
		holds true because of the factor $ \sqrt a $ in front of the 
		double integral. In order to prove that 
		\[
		\lim_{a \to \infty}   P({^I _\mu T_0} < a )  = 1 
		\]
		we note that after a change of variable the double integral in 
		\eqref{eq:t0-iter-dist} can be written as 
		\begin{align*}
			&
			\frac{2 \sqrt a }{\pi}
			\int_{a}^{\infty} 
			\frac{
				e^{- \frac{\mu^2w}{2} }
			}{w \sqrt{w-a}} 
			P\Big(  \max_{0\leq z \leq t}|B(z)| <w  \Big)  \, \mathrm d w\\
			&=
			\frac{2 \sqrt a }{\pi}
			\int_{0}^{\infty} 
			\frac{
				e^{- \frac{\mu^2 (w+a) }{2} }
			}{(w + a) \sqrt{wa}} 
			P\Big(  \max_{0\leq z \leq t}|B(z)| <w + a  \Big)  \, \mathrm d w
			\xrightarrow[a\to \infty]{} 0 
		\end{align*}
		and the convergence to zero  can be checked by applying the dominated convergence theorem. 
		The distribution \eqref{eq:t0-iter-dist} simplifies for $ \mu =0 $ and becomes
		\begin{align*}
			P(^I T_0 <a ) &= 
			\frac{2 \sqrt a }{\pi}
			\int_{a}^{\infty} 
			\frac{
				\mathrm d w
			}{w \sqrt{w-a}} \, 
			\int_{-w}^{w} \frac{\mathrm d z}{\sqrt{2 \pi t}}
			\sum_{r=-\infty}^{\infty} (-1)^r e^{ - \frac{(z-2wr)^2}{2t} } \\
			&=
			\frac{2 \sqrt a }{\pi}
			\int_{a}^{\infty} 
			\frac{
				\mathrm d w
			}{w \sqrt{w-a}} \, 
			P\Big(  \max_{0\leq z \leq t}|B(z)| <w  \Big) \\
			&=
			P\Big(  \max_{0\leq z \leq t}|B(z)| <a  \Big) + 
			\int_a^{\infty} 
			\frac 2\pi \arcsin \sqrt \frac aw  \, P\Big(  \max_{0\leq z \leq t}|B(z)| \in \mathrm d w  \Big)
			\enspace .
		\end{align*}
	\end{remark}
\end{proof}
By applying the previous arguments we can easily arrive at the distribution of 
\[ {_{\mu_1}^{\mu_2} T}_{0,t} = \sup \{ s < \max_{0\leq z \leq t} B^{\mu_2} (z) : B^{\mu_1}(s) = 0 \}\]
by substituting the distribution of the maximum in 
\eqref{eq:t0-iter-dist} with that of a Brownian motion with drift 
$\mu_2$ given in \eqref{eq:max-drift-dist}.
\begin{remark}
	We now consider the random variable $ {^{\mu_1} T}_{0, {^{\mu_2} T}_{0,t}} $. It can be regarded as the last zero-crossing of the 
	Brownian motion $ B^{\mu_1} $ during the time interval $ (0, {^{\mu_2} T}_{0,t}) $, that is we observe the first Brownian motion
	up to the time when a second Brownian motion $ B^{\mu_2} $ has crossed the zero level for the last time before $ t $.
	The two drifted Brownian motions  $ B^{\mu_1} $, $ B^{\mu_2}\,, t>0 $  are again supposed to be independent. 
	\begin{align*}\label{eq:t0-t0-dist}
			&P( {^{\mu_1} T}_{0, {^{\mu_2} T}_{0,t}} <a)  \numberthis \\
			&=
			\int_{0}^{t} P( {^{\mu_1} T}_{0,w} <a ) \,  P( {^{\mu_2} T}_{0,t} \in \mathrm d w ) \\
			&=
			\int_0^a P( {^{\mu_2} T}_{0,t} \in \mathrm d w )  +
			\int_a^t P( {^{\mu_1} T}_{0,w} <a ) \,  \frac{\mathrm d}{\mathrm d w }P( {^{\mu_2} T}_{0,t} < w ) \\
			&=
			P( {^{\mu_2} T}_{0,t} <a ) + 
			P( {^{\mu_1} T}_{0,w} <a ) P( {^{\mu_2} T}_{0,t} <w )\big|_{w=a}^{w=t}  \\
			& \qquad - 
			\int_a^t \frac{\mathrm d}{\mathrm d w }  P( {^{\mu_1} T}_{0,w} <a ) \, P( {^{\mu_2} T}_{0,t} <w ) 
			\, \mathrm d w\\
			&=
			P( {^{\mu_1} T}_{0,t} <a ) -
			\int_{a}^{t} \mathrm d w 
			 e^{-\frac{\mu_1^2 w}{2}} \frac{\sqrt{a} }{\pi w \sqrt{w - a } }
			 \,\, \times  \\
			& \qquad \times 
			\bigg[
			1 - \frac 2\pi e^{  - \frac{\mu_2^2 w }{2}} \int_{0}^{\arccos \sqrt \frac wt} 
			e^{ - \frac{\mu_2^2 w }{2} \tan^2 \theta }\, \mathrm d \theta 
			\bigg] 	\enspace .
	\end{align*}
	Formula \eqref{eq:t0-t0-dist} simplifies for $ \mu_1 = \mu_2 =0 $ and in this case becomes
	\begin{align*}\label{eq:t0-t0-dist-0}
			P( T_{0, T_{0,t}} <a) &= 
			\frac 2 \pi \arcsin \sqrt \frac at + 
			\int_a^t \frac 2\pi \arcsin \sqrt \frac aw 
			\frac{\mathrm d w}{\pi \sqrt{w(t-w)}}  \numberthis\\
			&=
			P( T_{0,t} < a) +  \frac 2\pi \mathbb E \Big[ \arcsin \sqrt \frac a {T_{0,t}} \mathds{1}_{( T_{0,t}>a)}\Big]
	\end{align*}
\end{remark}
Considering the second line of \eqref{eq:t0-t0-dist} we can therefore write out 
the density of the iterated zero-crossing time as
\begin{align}\label{eq:t0-t0-dens}
	P (  {^{\mu_1}T}_{0, {^{\mu_2}T}_{0,t}}  \in \mathrm d a) = 
	\int_{a}^t P (  {^{\mu_1}T}_{0, w}  \in \mathrm d a) P( {^{\mu_2}T}_{0, t}  \in \mathrm d w )
\end{align}
which can be given explicitly in the driftless case  as
\begin{align}
	P( T_{0, T_{0,t}} \in \mathrm d a) /\mathrm d a
	= \int_{a}^{t} 
	\frac{\mathrm d w}
	{\pi \sqrt{ a (w-a) } \pi \sqrt{w(t-w)}} \qquad 0 < a< t \enspace .
\end{align}
We consider now $ n+1 $ independent Brownian motions $ B^{\mu_j}(t),\, 1\leq j\leq n+1 $
and define recursively the following random times 
\[   {^{\mu_1, \ldots \mu_n}T}_{0,t} 
= \sup \{ s < {^{\mu_{2}, \ldots \mu_n}T}_{0,t} : B^{\mu_1}(s) = 0\} \,\,.
\]
and
\begin{gather*}
	{^{\mu_j, \ldots \mu_n}T}_{0,t} 
	= \sup \{ s < {^{\mu_{j+1}, \ldots \mu_n}T}_{0,t} : B^{\mu_j}(s) = 0\} \,\,.
\end{gather*}
It is straightforward to generalize the reasoning which leads to \eqref{eq:t0-t0-dens} 
to the situation where the Brownian motion is iterated
an arbitrary number of times. In fact we can write that
\begin{align*}
		&P(  {^{\mu_1, \ldots \mu_n}T}_{0,t}  \in \mathrm d a) \numberthis \\
		&=
		\int_{a}^{t} \int_{w_1}^{t} \cdots \int_{w_{n-1}}^{t} 
		P ( {^{\mu_1}T}_{0,w_1}  \in \mathrm d a ) 
		P ( {^{\mu_2}T}_{0,w_2}  \in \mathrm d w_1 ) \\
		&\qquad \times 
		P ( {^{\mu_3}T}_{0,w_3}  \in \mathrm d w_2 ) \cdots 
		P ( {^{\mu_{n+1}}T}_{0,w_t}  \in \mathrm d w_{n} ) \enspace .
\end{align*}
For $ \mu_1 = \mu_2 = \ldots = \mu_{n+1} = 0 $
\begin{align*}\label{eq:t0-t0-n-0}
		&P(  {^{n}T}_{0,t}  \in \mathrm d a) \numberthis  \\
		&=
		\int_{a}^{t} \int_{w_1}^{t} \cdots \int_{w_{n-1}}^{t} 
		\frac{\mathrm d w_1}
		{\pi \sqrt{ a (w_1-a) } }
		\frac{\mathrm d w_1}
		{\pi \sqrt{ w_1 (w_2-w_1) } } \cdots
		\frac{\mathrm d w_n}
		{\pi \sqrt{ w_n (t-w_n) } } \enspace .
\end{align*}
We stress out the fact that $ n- $fold integral \eqref{eq:t0-t0-n-0} is carried out in the set
$ 0<a< w_1 < w_2 < \cdots < w_n<t $. Thus the mean value of the random time 
$ {^{\mu_1, \ldots \mu_n}T}_{0,t}  $
can be calculated as follows. 
\begin{align*}
		&\mathbb E {^{n}T}_{0,t} \numberthis  \\
		&= 
		\int_{0}^{t} a \, \mathrm d a 
		\int_{a}^{t} \int_{w_1}^{t} \cdots \int_{w_{n-1}}^{t} 
		\frac{\mathrm d w_1}
		{\pi \sqrt{ a (w_1-a) } }
		\frac{\mathrm d w_1}
		{\pi \sqrt{ w_1 (w_2-w_1) } } \cdots 
		\frac{\mathrm d w_n}
		{\pi \sqrt{ w_n (t-w_n) } } \\
		&=
		\int_{0}^{t} 
		\frac{\mathrm d w_n}
		{\pi \sqrt{ w_n (t-w_n) } } 
		\int_{0}^{w_n}
		\frac{\mathrm d w_{n-1}}
		{\pi \sqrt{ w_{n-1} (w_n- w_{n-1}) } } \cdots 
		\int_{0}^{w_1}
		\frac{a \, \mathrm d a}
		{\pi \sqrt{ a (w_1-a) } }
		\\
		&=
		\left[ 
		\frac{\Gamma \big( \frac 32\big)  \Gamma\big( \frac 12 \big)}{ \Gamma(2)}
		\right]^n \frac{t}{\pi^n} = \frac{t}{2^n} \enspace .
\end{align*}
Similarly we can give the $ m- $th order moment of $   {^{n}T}_{0,t} $ as
\begin{equation}
\mathbb E [ {^{n}T}_{0,t} ]^m = 
\left[ 
\frac{\Gamma \big( m +\frac 12\big)  \Gamma\big( \frac 12 \big)}{ \Gamma(m+1)}
\right]^n \frac{t^m}{\pi^n} = \left[ \binom{2m}{m} \frac{1}{2^{2m}}\right]^n {t^m} \enspace .
\end{equation}
The m.g.f. of $ {^{n}T}_{0,t} $ becomes 
\[ 
	M(\alpha) = \mathbb E e^{ \alpha ({^{n}T}_{0,t}) } = 
	\sum_{m=0}^{\infty} \left(\frac {\alpha t}{2^2}\right)^{\!m} \frac{  (2m!)^n  }{ (m!)^{2n+1} } \enspace .
\]
which is a pointwise convergent series as the ratio criterion shows. Furthermore from the Weierstrass $ M- $test we infer that 
$ M(\alpha) $ converges uniformly and absolutely on $ \alpha- $compact subsets. From the above calculations 
we easily infer that 
\[ 
\mathbb E ( {^{n}T}_{0,t} - 0  )^2 \xrightarrow[n\to \infty]{} 0
\]
since 
\[ \mathbb E ( {^{n}T}_{0,t} )^2 = \left( \frac 38\right)^n t^2  \enspace . \]

\subsection*{Acknowledgement}
The authors thank  the referee for his appreciation of our work and his suggestions which lead to an improved version of the paper.


	%
	\medskip
	\bibliographystyle{abbrvnat}
	\bibliography{biblio}

\begin{thebibliography}{10}
\providecommand{\natexlab}[1]{#1}
\providecommand{\url}[1]{\texttt{#1}}
\expandafter\ifx\csname urlstyle\endcsname\relax
  \providecommand{\doi}[1]{doi: #1}\else
  \providecommand{\doi}{doi: \begingroup \urlstyle{rm}\Url}\fi

\bibitem[DeBlassie(2004)]{deblassie2004}
R.~D. DeBlassie.
\newblock Iterated {Brownian} motion in an open set.
\newblock \emph{Ann. Appl. Probab.}, 14\penalty0 (3):\penalty0 1529--1558, 08
  2004.

\bibitem[Funaki(1979)]{funaki1979}
T.~Funaki.
\newblock Probabilistic construction of the solution of some higher order
  parabolic differential equation.
\newblock \emph{Proc. Japan Acad. Ser. A Math. Sci. 55}, 9\penalty0
  (2):\penalty0 176--179, Apr 1979.

\bibitem[Gradshteyn and Ryzhik(1980)]{gradtable}
I.~Gradshteyn and I.~Ryzhik.
\newblock \emph{Table of Integrals, Series, and Products}.
\newblock Academic Press, 1980.

\bibitem[Hochberg and Orsingher(1996)]{Hochberg1996}
K.~J. Hochberg and E.~Orsingher.
\newblock Composition of stochastic processes governed by higher-order
  parabolic and hyperbolic equations.
\newblock \emph{Journal of Theoretical Probability}, 9\penalty0 (2):\penalty0
  511--532, Apr 1996.
\newblock ISSN 1572-9230.

\bibitem[It\^o and McKean(1974)]{ito1974diffusion}
K.~It\^o and H.~McKean.
\newblock \emph{Diffusion Processes and Their Sample Paths}.
\newblock Grundlehren der Mathematischen Wissenschaften. Springer, 1974.

\bibitem[Khoshnevisan and Lewis(1996{\natexlab{a}})]{Khoshnevisan-lew1996}
D.~Khoshnevisan and T.~M. Lewis.
\newblock Chung's law of the iterated logarithm for iterated brownian motion.
\newblock \emph{Annales de l'I.H.P. Probabilites et statistiques}, 32\penalty0
  (3):\penalty0 349--359, 1996{\natexlab{a}}.

\bibitem[Khoshnevisan and Lewis(1996{\natexlab{b}})]{Khoshnevisan1996}
D.~Khoshnevisan and T.~M. Lewis.
\newblock The uniform modulus of continuity of iterated brownian motion.
\newblock \emph{Journal of Theoretical Probability}, 9\penalty0 (2):\penalty0
  317--333, Apr 1996{\natexlab{b}}.

\bibitem[Orsingher and Beghin(2009)]{orsingher09}
E.~Orsingher and L.~Beghin.
\newblock Fractional diffusion equations and processes with randomly varying
  time.
\newblock \emph{The Annals of Probability}, 37\penalty0 (1):\penalty0 206--249,
  2009.

\bibitem[Orsingher and D'Ovidio(2011)]{Orsingher2011}
E.~Orsingher and M.~D'Ovidio.
\newblock Vibrations and fractional vibrations of rods, plates and {F}resnel
  pseudo-processes.
\newblock \emph{Journal of Statistical Physics}, 145\penalty0 (1):\penalty0
  143, Sep 2011.

\bibitem[Shorack and Wellner(2009)]{shorack2009empirical}
G.~Shorack and J.~Wellner.
\newblock \emph{Empirical Processes with Applications to Statistics}.
\newblock Classics in Applied Mathematics. Society for Industrial and Applied
  Mathematics, 2009.
\newblock ISBN 9780898719017.

\end{thebibliography}
\end{document}